\pdfoutput=1
\documentclass[11pt,a4paper]{amsart}
\usepackage{amssymb}
\usepackage{amsthm}
\usepackage{amsmath}
\usepackage{amsxtra}
\usepackage{latexsym}
\usepackage{mathrsfs}
\usepackage[all,cmtip]{xy}
\usepackage[all]{xy}
\usepackage{enumitem}
\usepackage{xcolor}
\usepackage{graphicx}
\usepackage{comment}
\usepackage{mathabx,epsfig}
\usepackage{stmaryrd}
\usepackage{comment}
\usepackage{pst-node}
\usepackage{mathtools}

\usepackage{tikz-cd} 

\usepackage{hyperref}

\newcommand{\RNum}[1]{\uppercase\expandafter{\romannumeral #1\relax}}

\newtheorem{thm}{Theorem}[section]
\newtheorem{lem}[thm]{Lemma}
\newtheorem{conj}[thm]{Conjecture}

\newtheorem{prop}[thm]{Proposition}
\newtheorem{cor}[thm]{Corollary}

\theoremstyle{definition}
\newtheorem{defn}[thm]{Definition}
\newtheorem{rmk}[thm]{Remark}
\newtheorem{example}[thm]{Example}
\newtheorem{question}[thm]{Question}

\numberwithin{equation}{section}

\newcommand\be{\begin{equation}}
\newcommand\ba{\begin{eqnarray}}
\newcommand\ee{\end{equation}}
\newcommand\ea{\end{eqnarray}}

\def\C{{\mathbb C}}

\def\Z{{\mathbb Z}}
\def\P{{\mathbb P}}
\def\A{{\mathbb A}}
\def\N{{\mathbb N}}

\DeclareMathOperator{\Prep}{Prep}

\DeclareMathOperator{\Per}{Per}

\DeclareMathOperator{\Aut}{Aut}

\DeclareMathOperator{\Sym}{Sym}

\DeclareMathOperator{\Orb}{Orb}

  {\end{list}}

\title[Polynomial endomorphisms of $\A^2$ with many periodic curves]
{Polynomial endomorphisms of $\A^2$ with many periodic curves.}
\thanks{The author was supported in part by NSERC grant RGPIN-2022-02951.}

\author{Xiao Zhong}
\address{University of Waterloo \\
Department of Pure Mathematics \\
Waterloo, Ontario \\
Canada  N2L 3G1}
\email{x48zhong@uwaterloo.ca}
\date{\today}
\begin{document}
\maketitle
\begin{abstract}
   In this paper, we prove that for a regular polynomial endomorphism of positive degree on $\P^2$, a family of curves containing a Zariski dense set of periodic curves is invariant under some iterate of the endomorphism. The setting is closely related to the Relative Dynamical Manin-Mumford Conjecture, recently proposed by DeMarco and Mavraki, which concerns a parametrized family of endomorphisms and varieties. Our result proves a weaker version of the conjecture where the endomorphism is a regular polynomial endomorphism on $\P^2$ that remains fixed in the family, and the family of curves contains a dense set of periodic curves. This result can also be viewed as a Dynamical Manin-Mumford type statement on the moduli space of divisors, and it proves a special case of the Dynamical Manin-Mumford Conjecture with a stronger assumption. 
   
   Moreover, our result specifically implies a uniform degree stabilization statement for a generic set of curves in a family under the transformation of a regular polynomial endomorphism. We demonstrate that a more general degree stabilization statement for a family of positive dimension subvarieties in $\P^K$ under the transformation of a family of endomorphisms is predicted by the Relative Dynamical Manin-Mumford Conjecture. We then prove that it is true when $K=2$ for families of regular polynomial endomorphisms under certain restrictions on the ramifications at the line at infinity. 
   
   Finally, we demonstrate an application of our result to classify all regular polynomial endomorphisms that admit infinitely many periodic curves of bounded degree.
\end{abstract}

\section{Introduction}
The Dynamical Manin-Mumford Conjecture is the dynamical generalization of the well-known Manin-Mumford Conjecture, which conjectures that a subvariety contains many periodic points of an endomorphism must be special with respect to this endomorphism. After a series of study and refinement by \cite{Zhang1}, \cite{Zhang2} and \cite{GT21}, the precise statement of the conjecture is the following:

\begin{conj}\label{conj: DMM}
     Let
 $f :X\to X$ be a polarized endomorphism of a smooth projective variety over a field
 of characteristic zero, and $Z\subset X$ be a subvariety containing a Zariski dense set of
 preperiodic points. Then either $Z$ is preperiodic or $Z$ is special, in the sense that it is
 contained in some subvariety $Y$ that is both $f^n$-and $\psi$-invariant, for some $n \geq 1$, where
$\psi$ is another polarized endomorphism commuting with $f^n$ on $Y$, and $Z$ is preperiodic under
 $\psi$.
 \end{conj}
The conjecture remains widely open, and only a few cases are known. For example, the conjecture is proven for splitting morphisms on $(\P^1)^n$ by \cite{GNY18},\cite{GNY19} and for polynomial endomorphisms on $\A^2$ extendable to an endomorphism on $\P^2$ with some constraint on the ramification at the line at infinity by \cite{DFR23}.
\subsection{Relative Dynamical Manin-Mumford Conjecture}
Recently, inspired by Gao and Habegger's work on Relative Manin-Mumford Conjecture for families of abelian varieties \cite[Theorem 1.1 and 1.3]{GH23}, DeMarco and Mavraki proposed the generalized Dynamical Manin-Mumford Conjecture for families of dynamical system in \cite{DM24}. 

We follow \cite{DM24} to introduce the necessary notations. An algebraic family of endomorphisms of $\P^n$ of degree $d$ is a morphism $$\Phi:  S \times \P^n \to S \times \P^n$$
given by $\Phi(s,z) =(s, f_s(z)) $ where $f_s$ is an endomorphism of $\P^n$ of degree $d$. Let $\mathcal{X} \subseteq S \times \P^n$ denote a closed irreducible subvariety which is flat over a Zariski open subset of $S$. We use $\mathbf{X}$ denote the generic fiber of $\mathcal{X}$ and let $\mathbf{\Phi} : \mathbf{P}^n \to \mathbf{P}^n$ be the map induced by $\Phi$, viewed as an endomorphism over the function field $\C(S)$.

We say $\mathcal{X}$ is $\Phi$-special if there exists
 a subvariety $\mathbf{Z} \subseteq \mathbf{P}^n$ over the algebraic closure $\C(S)$ containing the generic fiber $\mathbf{X}$, a
 polarizable endomorphism $\mathbf{\Psi} : \mathbf{Z} \to \mathbf{Z}$, and a positive integer $n$ such that the following hold:
\begin{itemize}
    \item $\mathbf{\Phi}^n(\mathbf{Z}) = \mathbf{Z}$;
    \item $\mathbf{\Phi}^n \circ \mathbf{\Psi} = \mathbf{\Psi} \circ \mathbf{\Phi}^n$ on $\mathbf{Z}$; and
    \item $\mathbf{X}$ is preperiodic under $\mathbf{\Psi}$.
\end{itemize}

We denote $r_{\Phi,\mathcal{X}}$ the relative spacial dimension of $\mathcal{X}$ over $S$. This is given by 
$$ r_{\Phi, \mathcal{X}} \coloneqq \min\{\dim_S\mathcal{Y} : \mathcal{X} \subseteq \mathcal{Y} \text{ and } \mathcal{Y} \text{ is $\Phi$-special}\},$$
where $\dim_S \mathcal{Y} = \dim \mathcal{Y} - \dim S$ is the dimension of a generic fiber of $\mathcal{Y}$ over $S$.

With the notations from above, DeMarco and Mavraki proposed the following relative version of Dynamical Manin-Mumford Conjecture: 
\begin{conj}\label{conj: DM24-conj-1,1}
     Let $\Phi : S \times \P^N \to S \times \P^N$ be an algebraic family of morphisms of degree $>1$, and let $\mathcal{X} \subseteq S \times \P^N$ be a complex, irreducible subvariety which is flat over $S$. The
 following are equivalent:
 \begin{itemize}
     \item $\mathcal{X}$ contains a Zariski-dense set of $\Phi$-preperiodic points.
     \item $\hat{T}^{r_{\Phi.\mathcal{X}}}_\Phi \wedge [X] \neq 0 $ for the relative special dimension $r_{\Phi,\mathcal{X}}$.
     \end{itemize}

\end{conj}
 Here $\hat{T}_\Phi$ is the canonical Green current associated to $\Phi$ on $S \times  \P^N$.

This is a very ambitious conjecture and remains open in most cases. When $\dim S = 0$, the Conjecture \ref{conj: DM24-conj-1,1} reduces to the Dynamical Manin-Mumford Conjecture. When $N = 1$, the Conjecture is already known. Moreover, several important problems can be viewed as special cases of Conjecture \ref{conj: DM24-conj-1,1}, including the Dynamical Andr\'e-Oort Conjecture proposed by \cite{BD13} and \cite{GHT15}, and questions on the uniform bounds for common preperiodic points between two rational functions on $\P^1$, see for example \cite{DM24-2}. See \cite[Section 3]{DM24} for detailed discussion.

The recent work of Mavraki and Schmidt \cite{MS24} establish a special case of a weaker version of the conjecture if one replaces $\P^N$ with $(\P^1)^N$ (see also the discussion in \cite[Conjecture 1.2]{DM24}).
\subsection{Main Result}
In this paper, we study {\bf regular polynomial endomorphisms} of $\P^2$ of degree $>1$, that is, endomorphisms of $\P^2$ extending polynomial endomorphisms on $\A^2$. We focus on the {\em constant family} case, where
\[
\Phi(s,p) = (s, F(p))
\]
for a fixed regular polynomial endomorphism $F$. Within this setting, we consider a family $\mathcal{X}$ that contains a Zariski dense set of periodic curves. Our first main result shows that such a family must be invariant under $F$, and moreover, that the degree of a generic member of the family remains stable under the action of $F$.

\begin{thm}\label{thm: main-1}
     Let $Z \subseteq \mathcal{M}_d$ be a subvariety of the space of effective divisors of degree $d$ in $\P^2$. Let $F$ be a regular polynomial endomorphism on $\P^2$ of degree $>1$. Suppose that $Z$ contains a Zariski dense set of divisors corresponding to periodic curves under $F$. Then, after replacing $F$ with some iterate, for a generic curve $C$ such that $[C] \in Z$, we have $\deg(C) = \deg(F(C))$. Moreover, for any curve $C'$ such that $[C'] \in Z$ we have $[F(C')] \in Z$.
 \end{thm}
\begin{rmk}
  Suppose $Z$ is a subvariety of the Chow variety of pure dimension $D$ and degree $d$ cycles inside $\P^K$, for some $K \geq 2$. Throughout the article, for a subvariety $C \subseteq \P^K$ of dimension $D$, we let $[C]$ denote the algebraic cycle corresponding to $C$, and we write $[C] \in Z$ if there exists a $z \in Z$ whose geometric support as a variety in $\P^K$ is equal to the union of the irreducible components of $C$.
\end{rmk}
 A key ingredient in the proof of Theorem~\ref{thm: main-1} builds on ideas from~\cite[Sections~5--6]{Xie23}. 
These techniques play a central role in handling invariant curves through an unramified point on the line at infinity and form the foundation of our approach.

 Theorem \ref{thm: main-1} implies a special case of Dynamical Manin-Mumford Conjecture under a stronger assumption (see Remark \ref{rmk: special-case-DMM}), and also proves a weaker form of the Conjecture \ref{conj: DM24-conj-1,1}:
 \begin{cor}\label{cor: conj-1.1-DM}
      Let $S$ be a smooth and irreducible quasi-projective variety defined over $\C$. Let $\Phi : S \times \P^2 \to S \times \P^2 $ be a constant family of endomorphism such that for every point $(s, p) \in S \times \P^2$
     $$ \Phi(s,p) = (s, F(p))$$
     for a regular polynomial endomorphism $F$ of degree $>1$ does not depend on $s$. Let $\mathcal{X} \subseteq S \times \P^2$ be an irreducible hypersurface that is flat over $S$ and projects dominantly to $S$. Suppose there exists a Zariski dense set of $s \in S$ such that $X_s$, the fiber of $\mathcal{X}$ at $s$, is a periodic curve under $F$. Then for all $N \in \Z^+$, we have 
     $$ \hat{T}_{\Phi^{\times N}}^{\wedge r_{\Phi^{\times N}, \mathcal{X}^N}} \wedge [\mathcal{X}^N] \neq 0,$$
     and $r_{\Phi^{\times N}, \mathcal{X}^N} \leq \min\{ 2N , N + \dim S\}$.
 \end{cor}
    \begin{rmk}
    Throughout this paper, we will let $\mathcal{X}^N$ and $\Phi^{\times N}$ denote the $N$-th fiber product of $\mathcal{X}$ and $\Phi$ over $S$.     
    \end{rmk}

    As another consequence of Theorem \ref{thm: main-1}, we combined it with the classification of endomorphisms on $\P^2$ preserving an algebraic $k$-webs obtained in \cite{FP15} to describe all regular polynomial endomorphisms that admit infinitely many periodic curves of bounded degrees.
\begin{thm} \label{thm: classification-summary}
    Let $D$ be any positive integer. Let $F$ be a regular polynomial endomorphism on $\P^2$ of degree $>1$ and suppose that it admits infinitely many periodic curves of degree bounded by $D$. Then, after replacing $F$ with some iterate, it is of one of the following forms:
    \begin{itemize}
        \item $F$ is a polynomial skew product;
        \item $F$ is homogeneous;
        \item there exists a generic finite rational map $\mu : \P^1 \times \P^1 \to \P^2$ and a split polynomial endomorphism $(f,g)$ on $\P^1 \times \P^1$ such that $$\mu \circ (f,g) = F \circ \mu .$$
    \end{itemize}
\end{thm}
A more detailed description of the third case is given in the last section. 
This is an extension of the work in \cite[Section 6]{Xie23} where Xie proved that for a regular polynomial endomorphism whose restriction on the line at infinity is not a polynomial, the existence of infinitely many periodic curves forces the map to be homogeneous \cite[Theorem 6.2]{Xie23}.

\subsection{Some Conditional Results on Degree Stabilization}

Theorem \ref{thm: main-1} can also be interpreted as a Dynamical Manin-Mumford type statement on the moduli space of divisors: instead of looking at subvareities with a dense set of preperiodic points in $\P^2$, we look at the subvarieties in the moduli space $\mathcal{M}_d$ containing a dense set of periodic curves under $F$. The theorem asserts that such a subvariety must be special in the sense that a generic divisor it contains has stabilized degree under the action of $F$, and the subvariety itself is invariant under the induced map.

We expect the degree stabilization should hold in a more general setting, and we propose the following question which was raised in a discussion with Junyi Xie: 
\begin{question}\label{question: degree-stable}
    Let $F$ be an endomorphism of degree $>1$ on $\P^N$ and $Ch_{D,d}(\P^N)$ be the Chow variety consisting of algebraic cycles of $\P^N$ of dimension $D$ and degree $d$. Suppose $Z \subseteq Ch_{D,d}(\P^N)$ is a subvariety containing a Zariski dense set of algebraic cycles corresponding to irreducible periodic subvarieties of $\P^N$ of degree $d$ under $F$. Then is it true that, after replacing $F$ with some iterate, for a generic cycle $V \in Z$, we have $\deg(V) = \deg(F(V))$?
\end{question}
\begin{example}
    Note that even if a curve $C \subseteq \P^2$ is periodic under an endomorphism $F$, the curves in the periodic cycle can have various degrees. For example, under the regular polynomial endomorphism $F(x,y) = (x^2 , y^2 -x)$, we have the following periodic cycle of curves
    $$ \begin{tikzcd}
        V(y) \arrow{r}{F}  &  V(y^2 - x)\arrow{r}{F} & V(y)
\end{tikzcd}  $$
\end{example}

  With a mild assumption on $F$, we will demonstrate that the Dynamical Manin-Mumford Conjecture implies a positive answer to the Question \ref{question: degree-stable}, and a broader version of Theorem \ref{thm: main-1} should hold:
  \begin{defn}
Let $S$ be a smooth, irreducible, quasi-projective variety over $\C$ of positive dimension, and let $M \ge 2$ be an integer.  
Let 
\[
\Phi : S \times \P^M \to S \times \P^M
\]
be a family of surjective endomorphisms, and let $\mathcal{X} \subseteq S \times \P^M$ be an irreducible subvariety of dimension greater than $\dim S$, which projects dominantly onto $S$ and is flat over $S$.  

We say that $(\Phi, \mathcal{X})$ is of \textbf{general type} if for any subvariety $\mathbf{ Y} \subseteq \mathbf{P}^M$ containing $\mathbf{X}$ and invariant under $\mathbf{\Phi}^n$ for some $n > 0$, the only polarizable endomorphisms on $\mathbf{P}^M$ that commute with $\mathbf{\Phi}^n$ on $\mathbf{Y}$ are the iterates of $\Phi$ itself.

In the special case where $\Phi(s, p) = (s, F(p))$ for a fixed surjective endomorphism $F$, we also say that $(F, \mathcal{X})$ is of general type, meaning the same as $(\Phi, \mathcal{X})$ being of general type.
\end{defn}

\begin{thm}\label{thm: DMM-implies-general-main}
    Let $F$ be a surjective endomorphism on $\P^K$ of degree $>1$. Let $Z \subseteq Ch_{D,d}(\P^K)$ be a subvariety in the Chow variety consisting of algebraic cycles of $\P^K$ of dimension $D \leq K-1$ and degree $d$. Suppose $Z$ contains a Zariski dense set of algebraic cycles corresponding to irreducible periodic subvarieties of $\P^K$ under $F$ of degree $d$. 
    Assuming $(F,Z)$ is of general type and Conjecture \ref{conj: DMM} is true, then for every subvariety $C \subseteq \P^K$ and a generic subvariety $C' \subseteq \P^K$ such that $[C], [C'] \in Z$, we have that $[F(C)] \in Z$ and $\deg(C') = \deg(F(C'))$, after replacing $F$ with some iterate.
\end{thm}

Furthermore, we show that Conjecture \ref{conj: DM24-conj-1,1} implies an even stronger stabilization property for the degrees of images in families of endomorphisms:
\begin{prop}\label{conj: cor-degree-stable-family}
     Let $S$ be a smooth and irreducible quasi-projective variety defined over $\C$ of positive dimension and $M \geq 2$ be a positive integer. Let $\Phi : S \times \P^M \to S \times \P^M$ be a family of endomorphism such that for every point $(s,p) \in S \times \P^M$ 
     $$\Phi(s,p) = (s, F_s(p))$$
     for an endomorphism $F_s$ of degree $>1$. Let $\mathcal{X} \subseteq S \times \P^M$ be an irreducible subvariety of dimension $> \dim S$ that projects dominantly to $S$ and is flat over $S$. Assume $(\Phi,\mathcal{X})$ is of general type.
     
     Suppose that for a Zariski dense subset of $s \in S$, $\mathcal{X}_s$ is preperiodic under $\Phi_s$. If Conjecture \ref{conj: DM24-conj-1,1} holds, then there exists positive integers $n,m$ and an infinite set of positive integers $I$ such that for any pair of integers $k_1, k_2 \in I$ we have that, $$\deg(F^{m + nk_1}_s(\mathcal{X}_s)) = \deg(F^{m + nk_2}_s(\mathcal{X}_s)),$$
     for a generic point $s \in S$.
\end{prop}



We also establish a degree stabilization result for families 
$\Phi : \mathbb{A}^1 \times \mathbb{P}^2 \to \mathbb{A}^1 \times \mathbb{P}^2$ 
of regular polynomial endomorphisms, under suitable assumptions on the behavior of $\mathcal{X}$ at infinity (Theorem \ref{thm: line-at-infinty-poly} and \ref{thm: line-at-infinity-algebraic}). 
Our argument combines analysis of branches along $H_\infty$ with results from~\cite{FG22} 
and, conditionally, on~\cite[Conjecture~6.1]{De16} and the classification of semiconjugate pairs in~\cite{Pa23}.

\section{Outline of the Paper}

Section~\ref{sect: preli} establishes several technical lemmas that will be essential for the subsequent arguments.

In Section~\ref{sect: proof-of-main}, we prove Theorem~\ref{thm: main-1}. 
A key technique used in the proof of Theorem \ref{thm: main-1} is inspired by \cite[Section 5 and 6]{Xie23}. It establishes that, for a regular polynomial endomorphism the invariant algebraic curve passing through a non-superattracting fixed point of $F| _{H_\infty}$ is unique whenever it exists, where $H_\infty = \P^2 \setminus \A^2$. Moreover, the intersection must be transverse. A slight adaptation of this result (Lemma \ref{lem: transvers-non-super-E}), combined with suitable blow-up procedures along $H_\infty$, 
allows us to obtain a reparametrization of a family of divisors containing infinitely many periodic curves. 
This forms a crucial step in the proof of Theorem~\ref{thm: main-1}.

We also explain at the end of this section how the Dynamical Manin–Mumford Conjecture suggests that the same expectation should hold for endomorphisms on $\P^K$ in general.

In Section~\ref{sect: impl-RDMM}, we show that our main result, Theorem~\ref{thm: main-1}, implies a weaker version of the \emph{relative Dynamical Manin–Mumford Conjecture} (Cojecture~\ref{conj: DM24-conj-1,1}) proposed by DeMarco and Mavraki.

Section~\ref{sect: degree-stabilization} contains conditional results. 
We show that the relative Dynamical Manin–Mumford Conjecture predicts a degree stabilization phenomenon for families of endomorphisms and subvarieties containing a Zariski-dense set of periodic subvarieties. 
We verify this expectation for families 
$\Phi : \mathbb{A}^1 \times \mathbb{P}^2 \to \mathbb{A}^1 \times \mathbb{P}^2$ 
of regular polynomial endomorphisms, under unramification assumptions on the behavior of $\mathcal{X}$ at infinity. 
Our argument relies on identifying when a marked point 
$a(t) \in \mathbf{X} \cap H_\infty$ 
becomes superattracting for infinitely many parameters $t_0 \in \A^1$. 
In the polynomial case, this follows from~\cite{FG22} (see Theorem~\ref{thm: line-at-infinty-poly}); in the general case, the proof depends conditionally on~\cite[Conjecture~6.1]{De16} and the argument relies on the classification of semiconjugate pairs in~\cite{Pa23} (see Theorem~\ref{thm: line-at-infinity-algebraic}).

Finally, Section~\ref{sect: classification} introduces the notion of $k$-webs following~\cite{FP15}. 
Using Theorem~\ref{thm: main-1}, we show that a regular polynomial endomorphism admitting infinitely many periodic curves of bounded degree must preserve either an algebraic $k$-web or a pencil of curves. 
Combining this with the classification results of~\cite{FP15} and~\cite{DJ07} yields a complete description of such endomorphisms, proving Theorem~\ref{thm: classification-summary}.

\section{Preliminaries}\label{sect: preli}
In this section, we collect some technical lemmas that will be useful in later proofs.

We begin by clarifying the notion of a resolution tree:
\begin{defn}
    By a {\bf resolution tree of depth $N$ at $p$ with respect to a curve $C$}, where $p \in \P^2$, $N \in \Z^+$, and $C \subset \P^2$ passes through $p$, we mean the tree obtained by iteratively blowing up points where the strict transform of $C$ intersects the exceptional divisors. The root of the tree is $p$ (level $0$), and the process is repeated until the tree has depth $N$.
\end{defn}

\begin{lem}\label{lem: finite-blows-up-deter}
    Let $D$ be a positive integer. If $C_1$ and $C_2$ are two affine curves in $\A^2$ passing through $(0,0) $ of degree $\leq D$ that have the same resolution tree of depth $D^2 +1$ at the point $(0,0)$, then $C_1 = C_2$.
 \end{lem}
 \begin{proof}

     Let $\pi_i$, $i \in \{1,2, \dots, D^2+1\}$ denote the (sequence of) blows-up map at $(0,0)$ which gives us the resolution tree of depth $i$ at $(0,0)$ with respect to both $C_1$ and $C_2$, as they share the same resolution tree.

     We show inductively that for any positive integer $M \leq D^2+1$, we have 
     $$ (C_1, C_2)_{(0,0)} \geq M + (\tilde{C}^{(M)}_1, \tilde{C}^{(M)}_2)_{p_M}$$
     where $\tilde{C}^{(M)}_i$ is the strict transformation of $C_i$ under $\pi_M$, $i \in \{1,2\}$ and $p_M$ is a point in the intersection of $\tilde{C}^{(M)}_1$, $\tilde{C}^{(M)}_2$ and an exceptional divisor introduced when we construct the $M$-th level of the resolution tree. The existence of such a $p_M$ is guaranteed by our assumption that $C_1$ and $C_2$ share the resolution tree up to depth $D^2+1$.
     
     Note that $\pi_1$ is the first blow up map at $(0,0)$, which introduces an exceptional divisor $E_1$ in the preimages of $(0,0)$ under $\pi_1$. Then we have 
     $$ \pi_1^* C_i = \tilde{C}^{(1)}_i + m_i E_1$$
     where $m_i > 0$ is the multiplicity of $E_1$ in the total transformation of $C_i$ under $\pi_1$, for $i \in \{1,2\}$. 
     Then, since $\tilde{C}^{(1)}_i \cdot E_1 = m_i$ for $i \in\{1,2\}$, we have
     \begin{align} \label{eq: blow-up-lemma-eq-1}        
     (C_1 \cdot C_2)_{(0,0)} &= (\tilde{C}^{(1)}_1 \cdot \tilde{C}^{(1)}_2)_{E_1} + m_2\tilde{C}^{(1)}_1 \cdot E_1 + m_1 \tilde{C}^{(1)}_2 \cdot E_1 + m_1m_2 E_1 \cdot E_1 \nonumber \\
      &= (\tilde{C}^{(1)}_1 \cdot \tilde{C}^{(1)}_2)_{E_1} + m_1m_2 \geq 1 + (\tilde{C}^{(1)}_1 \cdot \tilde{C}^{(1)}_2)_{E_1},\end{align}
     where 
     $$ (\tilde{C}^{(1)}_1 \cdot \tilde{C}^{(1)}_2)_{E_1} \coloneqq \sum_{p \in \tilde{C}^{(1)}_1 \cap \tilde{C}^{(1)}_2 \cap E_1} (\tilde{C}^{(1)}_1 \cdot \tilde{C}^{(1)}_2)_{p}.$$ This verifies the base case.

     Now, we assume that the statement holds for a positive integer $M < D^2 + 1$ and we show that it also holds for $M + 1$. 
     We do the further expansion of $$(\tilde{C}^{(M)}_1, \tilde{C}^{(M)}_2)_{p_M}$$ by taking a blow-up $\pi_{p_M}$ at $p_M$. 
Then, we have 
\begin{align}\label{eq: blow-up-inductive-1}
    &(\tilde{C}^{(M)}_1, \tilde{C}^{(M)}_2)_{p_M}  \nonumber\\&= (\tilde{C}^{(M),p_M}_1 \cdot  \tilde{C}^{(M),p_M}_2)_{E_{p_M}} + m^{p_M}_1 \tilde{C}^{(M),p_M}_2 \cdot E_{p_M} \nonumber \\&+ m^{p_M}_2 \tilde{C}^{(M),p_M}_1 \cdot E_{p_M} + m^{p_M}_1m^{p_M}_2 E_{p_M}\cdot E_{p_M}&\nonumber\\
          &= (\tilde{C}^{(M),p_M}_1 \cdot  \tilde{C}^{(M),p_M}_2)_{E_{p_M}} + m^{p_M}_1m^{p_M}_2,&
\end{align} 
where $$ \pi^*_{p_M} \tilde{C}^{(M)}_i =\tilde{C}^{(M),p_M}_i + m^{p_M}_i E_{p_M} ,$$
$m^{p_M}_i > 0$, $i \in \{1,2\}$ and $E_{p_M}$ is the exceptional divisor introduced by $\pi_{p_M}$.

Notice that the assumption that $C_1$ and $C_2$ share the resolution tree at $(0,0)$ up to depth $D^2+1$ implies that there exists a $$p_{M+1} \in \tilde{C}^{(M),p_M}_1 \cap \tilde{C}^{(M),p_M}_2 \cap E_{p_M}.$$

Also, there exists a sequence of blows-up $\pi'_{M+1}$ such that 
$$ \pi_{M+1} = \pi'_{M+1} \circ \pi_{p_M} \circ \pi_M$$
and 
$\tilde{C}^{(M+1)}_i$ is the strict transformation of $\tilde{C}^{(M),p_M}_i$ under $\pi'_{M+1}$, where $i \in \{1,2\}$. 

Then 
\begin{align}\label{eq: blow-up-inductive-2}
    (\tilde{C}^{(M),p_M}_1 \cdot  \tilde{C}^{(M),p_M}_2)_{E_{p_M}} \geq(\tilde{C}^{(M),p_M}_1 \cdot  \tilde{C}^{(M),p_M}_2)_{p_{M+1}} \nonumber\\ = ((\pi'_{M+1})^*\tilde{C}^{(M),p_M}_1 \cdot  (\pi'_{M+1})^*\tilde{C}^{(M),p_M}_2)_{p_{M+1}} \geq (\tilde{C}^{(M+1)}_1 \cdot  \tilde{C}^{(M+1)}_2)_{p_{M+1}}.
\end{align} 
     
     Now the induction hypothesis 
     $$ (C_1, C_2)_{(0,0)} \geq M + (\tilde{C}^{(M)}_1, \tilde{C}^{(M)}_2)_{p_M}$$
     together with (\ref{eq: blow-up-inductive-1}) and (\ref{eq: blow-up-inductive-2}) imply that 
     $$ (C_1, C_2)_{(0,0)} \geq M+1 + (\tilde{C}^{(M+1)}_1 \cdot  \tilde{C}^{(M+1)}_2)_{p_{M+1}}. $$

     Then the induction concludes the statement. Hence, plug in $M = D^2+1$, we have 
    $$ (C_1, C_2)_{(0,0)}  \geq  D^2 + 1 .$$
    Since, $\deg(C_1) = \deg(C_2) = D$, by Bézout’s intersection theorem we conclude that $C_1 = C_2$.

 \end{proof}

 \begin{lem}\label{lem: infinite-intersect-with-E}
    Let $\mathcal{C}$ be an infinite set of curves in $\P^2$ of uniformly bounded degree $\leq D$ passing through a single point $p \in \P^2$. There exists a blow up $\pi$ of length bounded by $D^2+1$ at the point $p$ and an exceptional divisor $E \subseteq \pi^{-1}(p)$ such that $\bigcup_{C \in \mathcal{C}} \pi^\# C \cap E$ contains infinitely many distinct points in $E$.
 \end{lem}
 \begin{proof}
      Suppose, for the sake of contradiction, that for any blow up $\pi$ at $p$ of length not greater than $D^2+1$ and any exceptional divisor $E$ in $\pi^{-1}(p)$ we have $|\bigcup_{C \in \mathcal{C}} \pi^\# C \cap E | < \infty$. We build a resolution tree of depth $D^2+1$ at $p$ by recursively blowing up all the points in the intersections of the strict transformation of $C \in \mathcal{C}$ with exceptional divisors. By our assumption, there are only finitely many exceptional divisors in this resolution tree and each exceptional divisors only intersects with $\mathcal{C}$ at a finite set of points. Therefore, by the Lemma \ref{lem: finite-blows-up-deter} and the pigeonhole principle, there are only finitely many distinct curves in $\mathcal{C}$, which is a contradiction. Thus, there must exists a divisor $E \subseteq \pi'^{-1}(p)$, with a blow up $\pi'$ at $p$ of length not greater than $D^2+1$ (pick a branch of the resolution tree), such that $\bigcup_{C \in \mathcal{C}} \pi'^\# C \cap E$ is an infinite set.
 \end{proof}

The following lemma is an adaptation of \cite[Lemma 5.11]{Xie23} for our purposes.  The proof follows the original almost verbatim; we will explain here the necessary changes.
\begin{lem}\label{lem: transvers-non-super-E}
    Let $F$ be a regular polynomial endomorphism on $\A^2$ of degree $>1$. Let $X$ be a compactification of $\A^2$ which is either $\P^2$ or constructed by a sequence of blows-up of some points on $H_\infty  = \P^2 \setminus\A^2$. Let $E$ be an irreducible component of $X \setminus \A^2$ such that $F$, after being replaced with some iterate, extends to a finite morphism, $f$, on $E$. Let $o \in E$ be a non-superattracting fixed point of $f$. Then there exists a unique irreducible formal curve $\hat{C}$ at $o$ which is $F$-invariant and not contained in $E$. Moreover, it intersects $E$ transversely at $o$.
\end{lem}
\begin{proof}
     In \cite[Lemma 5.11]{Xie23}, the author proves the this lemma with $X = \P^2$ and $E = H_\infty$. But the approach only requires the data locally around $o$ and so can be adapted to this setting.
     
     Let $L = \C((t))$ with $t$-adic norm. Let $U$ be the affinoid subdomain of $X^{an}_L$, which, in the local affine chart around $o = (0,0)$, is given by $$ U \coloneqq\{(x,y) : |x| \leq |t|, |y| \leq |t|\}.$$ Since $F$ induces a finite morphism on $E$, we have that $F$ induces a well-defined endomorphism $$g|_U : U \to U$$ which fixes the hyperplane $Y = E^{an}_L \cap U$, where $E$ within the affine chart is given by $V(x)$ by a coordinate choice. Now, the rest of the proof follows exactly the same line as in \cite{Xie23}[Lemma 5.11] with $g|_U$, $U$ and $Y$ as above.
\end{proof}

\begin{rmk}
    The key ingredient of the proof of this lemma, as explained in \cite[Proof of Lemma 5.11]{Xie23}, is \cite[Theorem 8.3]{Xie25}, which is the main result of the Appendix A in the paper. In this particular case of Lemma \ref{lem: transvers-non-super-E}, it shows that $g|_U$ is semiconjugated to $g|_Y$. 
    
    Notice that Lemma \ref{lem: transvers-non-super-E} has a similar flavor to \cite[Theorem 2.1]{DFR23}, which shows that there exists a unique smooth analytic curve through $o$ which is transverse to $E$ and invariant under $F$.
\end{rmk}

\section{Proof of Theorem \ref{thm: main-1}} \label{sect: proof-of-main}
In this section, we prove Theorem~\ref{thm: main-1}. 
We begin with a series of lemmas that form the intermediate steps toward the proof. 
Throughout this section, let $F$ be a regular polynomial endomorphism of $\A^2$ of degree $> 1$.
\begin{lem}\label{lem: one-point-or-web}
    Suppose $F: \P^2 \to \P^2$ admits an infinite set of periodic curves $\mathcal{C}$. Then one of the following holds:
    \begin{enumerate}
        \item[(1)] There exists a point $p \in H_\infty$ such that infinitely many curves in $\mathcal{C}$ pass through $p$.
        \item[(2)] There exists an integer $e \in \{1, 2\}$ and an infinite subset $\mathcal{C}' \subseteq \mathcal{C}$ consisting of periodic curves of degree $e$, together with a finite morphism 
        \[
            \sigma : \P^1 \longrightarrow Z,
        \]
        where $Z$ is the Zariski closure of $\mathcal{C}'$ in the space of effective divisors of degree $e$, $\mathcal{M}_e$, such that:
        \begin{itemize}
            \item for a generic $C \in Z$, we have $\deg(F(C)) = \deg(C)$; and 
            \item for any curve $C' \subseteq \P^2$ with $[C'] \in Z$, we have $[F(C')] \in Z$.
        \end{itemize}
    \end{enumerate}
\end{lem}
\begin{proof}
   Suppose that for every point $p \in H_\infty$, only finitely many periodic curves of $F$ pass through $p$. 
    We will show that $(2)$ must then occur. 
    
    By \cite[Lemma~6.6 and Remark~6.7]{Xie23}, every periodic curve of $F$, except the line at infinity $H_\infty = \P^2 \setminus \A^2$, has at most two branches at infinity. 
    Since $F|_{H_\infty}$ has only finitely many superattracting periodic cycles, and each periodic point in $H_\infty$ lies on finitely many periodic curves by assumption, there must exist infinitely many periodic curves 
    \[
        \mathcal{C}' \subseteq \mathcal{C}
    \]
    such that each $C \in \mathcal{C}'$ meets $H_\infty$ only at non-superattracting periodic points of $F|_{H_\infty}$. 
    
    For such curves, Lemma \ref{lem: transvers-non-super-E} implies that each intersection $C \cap H_\infty$ is transverse. 
    Consequently, the intersection multiplicity $C \cdot H_\infty$ equals the number of branches of $C$ at infinity, which is at most two. 
    Hence, $\deg(C) \leq 2$ for infinitely many periodic curves $C \in \mathcal{C}$. 
    Moreover, the same reasoning applies to all curves lying in the periodic cycles of these curves, showing that $\deg(C) = \deg(F(C)) \in \{1,2\}$ for all such $C$.

    We now regard these periodic curves as points of the projective space $\mathcal{M}_e$ of effective divisors of degree $e$ on $\P^2$, where $e \in \{1,2\}$. 
    By enlarging $\mathcal{C}'$ if necessary, we may assume that it is closed under the action of $F$, i.e., it contains all curves in the periodic cycles of its members.

    Consider the incidence variety
    \[
        \Gamma = \overline{\{(C, x) : C \in \mathcal{C}',\, x \in C \cap H_\infty\}} 
        \subseteq \mathcal{M}_e \times H_\infty,
    \]
    and let $\pi_i$ denote the projection onto the $i$-th factor for $i \in \{1,2\}$. 
    Since non-superattracting periodic points are dense in $H_\infty \cong \P^1$, $\pi_2$ is surjective; 
    by Lemma \ref{lem: transvers-non-super-E}, $\pi_2$ is in fact bijective. 
    Also, $\pi_1 : \Gamma \to \mathcal{M}_e$ is a morphism of degree $e$. 
    We may therefore define
    \[
        \sigma := \pi_1 \circ \pi_2^{-1} : \P^1 \cong H_\infty \longrightarrow \mathcal{M}_e,
    \]
    which is a finite morphism with image $Z = \overline{\pi_1(\Gamma)}$. 
    Note that even without enlarging $\mathcal{C}'$, the same $Z$ is obtained after taking closure, since $Z$ is irreducible, parametrized by $\P^1$, and contains all members of $\mathcal{C}'$ and their periodic cycles.

    Finally, let $Z_t$ be the irreducible curve in $\P^2_{\C(t)}$ corresponding to $Z$, where $\C(t)$ is the function field of $H_\infty$. 
    For any $t_0 \in \C$ such that $Z_{t_0} \in \mathcal{C}'$, we have $F(Z_{t_0}) = Z_{f(t_0)}$, where $f = F|_{H_\infty}$. 
    As the set of such $t_0$ (corresponding to non-superattracting periodic points of $f$) is dense in $\P^1$, it follows that
    \[
        F(Z_t) = Z_{f(t)} \quad \text{for all } t.
    \]
    Consequently, for every curve $C'$ with $[C'] \in Z$, we have $[F(C')] \in Z$, completing the proof.

\end{proof}

Now, by Lemma~\ref{lem: one-point-or-web}, one of the following two situations occurs. Either we can construct a one-dimensional family of curves, parametrized by $\P^1$, that contains infinitely many periodic curves and whose generic member has stabilized degree; or there exists a point $p \in H_\infty$ through which infinitely many periodic curves pass. 


 It remains to handle the second case.

 \begin{prop}\label{prop: moduli-DMM-1-D}
     Let $\mathcal{C}$ be an infinite set of periodic curves under $F$ in $\P^2$ of degree $D$ passing through a point $p \in H_\infty$. Then there exists an infinite subset of curves $\mathcal{C}' \subseteq \mathcal{C}$ such that we have a parametrization $\tau : \P^1 \to Z$, where $Z$ is the Zariski closure of $\mathcal{C}'$ in the space of effective divisors of degree $\leq D$ in $\P^2$. Moreover, after replacing $F$ with a suitable iterate, for any curve $C$ such that $[C] \in Z$, we have that $[F(C)] \in Z$ and, in particular, for a generic curve $C'$ such that $[C'] \in Z$, we have $\deg(F(C')) = \deg(C')$.
 \end{prop}
 \begin{proof}
     We first replace $F$ with some iterate and so $p$ is fixed under $F$. Moreover, by Lemma \ref{lem: transvers-non-super-E}, it is a super-attracting fixed point under $F|_{H_\infty}$.
     By Lemma \ref{lem: infinite-intersect-with-E}, there exists a blow up $\pi_0$ of length at most $D^2+1$ at $p$ such that $\mathcal{C}$ intersect a divisor $E_0 \subseteq \pi^{-1}_0(p)$ at infinitely many points. This in particular implies that $ F|_{E_0}$, the induced map of $F$ on $E$ after replacing $F$ with some iterate, is a finite morphism as curves in $\mathcal{C}$ are periodic and so are their intersections with $E_0$. Then there exists an infinite subset $\mathcal{C}' \subseteq \mathcal{C}$ such that for every curve in $\mathcal{C}'$, the intersection $C \cap E_0$ contains some non-superattracting periodic points of $F|_{E_0}$. This is because $F|_{E_0}$ is a finite endomorphism on $E_0$ and thus has only finitely many super-attracting periodic points. 
     
     Now, if there exists another point $p_1 \in \pi^{-1}_{0}(p)\cup H_\infty$ such that infinitely many curves from $\{\pi^\#_0 C : C \in \mathcal{C}'\}$ pass through $p_1$. Then we again use Lemma \ref{lem: infinite-intersect-with-E} to obtain a blow up $\pi_1$ at $p_1$, so that there exists a $E_1 \subseteq \pi_1^{-1}(p_1)$ whose intersection with strict transformations of curves in $\mathcal{C}'$ at infinitely many distinct points. Similarly, shrink $\mathcal{C}'$ and abuse notation to also let $\mathcal{C}'$ denote the infinite subset of $\mathcal{C}'$ only containing those curves with strict transformations intersecting $E_1$ at some non-superattracting periodic points of $F|_{E_1}$, after replacing $F$ with some iterate. 
     
     Now, we repeat this process if we can still find a point in $(\pi_1 \circ \pi_0)^{-1} (H_\infty) $ such that infinitely many strict transformations of curves in $\mathcal{C}'$ pass through it. Then, after the process stop, we obtain a sequence of blow-ups $\pi_0, \cdots, \pi_r$ and exceptional divisors $E_0, \dots, E_r$, such that there exists an infinite set of periodic curves $\mathcal{C}'$ whose strict transformations only intersect $E_i$, $i\in \{1,2 , \dots, r\}$, at non-superattracting periodic points of $F|_{E_i}$, where $r \in \Z^+$. 
     
     This process must terminate after at most $D$ iterations, and $r +1\leq D$, as, denoting $\pi = \pi_r \circ \cdots \circ \pi_0$, $$D \geq \deg(C) \geq (\pi^\# C, E_0 + \cdots + E_r) \geq r+1,
     $$
     for any $C \in \mathcal{C}'$. If the process did not stop, we would obtain $r + 1 = D + 1$, contradicting $\deg(C) \leq D$.

     Hence, after finitely many steps, we obtain a blow-up $\pi$ above $H_\infty$ such that through any point of $\pi^{-1}(H_\infty)$ only finitely many strict transforms $\pi^{\#}C$ (with $C \in \mathcal{C}'$) pass. Excluding from $\mathcal{C}'$ those finitely many curves whose strict transforms meet $E \subseteq \pi^{-1}(H_\infty)$ at some superattracting periodic points, we may assume that for all $C \in \mathcal{C}'$ and all components $E \subseteq \pi^*H_\infty$, the intersection $\pi^{\#}C \cap E$ consists only of non-superattracting periodic points of $F|_E$. We further choose an infinite subset of $\mathcal{C'}$ so that the strict transformation of every curve in it intersects every exceptional divisor (including $H_\infty$) at the same number of points. We abuse notation to still call it $\mathcal{C}'$.

       Now, notice that
     $$ \deg(C) = (\pi^\# C, \pi^*H_\infty),$$
     and by Lemma \ref{lem: transvers-non-super-E} we have every intersection between $\pi^\# C $ and every $E \subseteq \pi^* H_\infty$ is transverse (if non-empty) and so is the intersection between $\pi^\#(F(C))$ and $E$, since the intersection of $\pi^\# C$ and $E$'s are non-superattracting periodic points. Notice that there is a one-to-one correspondence between them as they live in the same periodic cycles under $F|_{E}$. Thus,
\[
(\pi^\# F(C), E) = (\pi^\# C, E),
\]
for every exceptional divisor \( E \subseteq \pi^* H_\infty \).
     Hence, we have 
     \[\deg(F(C)) =  (\pi^\# F(C), \pi^*H_\infty) =  (\pi^\# C, \pi^*H_\infty) = \deg(C).\]

     

    From now on, fix $E \coloneqq E_0$. 
    From now on, fix $E \coloneqq E_0$ and by our construction the strict transformation of every curve in $\mathcal{C}'$ intersects $E$ at $D' \leq D$ points. Define
    \[
        \mathcal{C}_0 = \{F^n(C) : C \in \mathcal{C}',\, n \ge 0\}
    \]
    and let $Z'$ be the closure of $\mathcal{C}_0$ in $\mathcal{M}_D$.

     Consider
    \[
        W = \{(p, C) \in E \times Z' : p \in \pi^{\#}C \cap E\} \subseteq E \times Z'.
    \]
    By Lemma~\ref{lem: transvers-non-super-E}, through each non-superattracting periodic point of $F|_E$ there passes a unique periodic curve (other than $E$ itself) intersecting $E$ transversely.
    Thus, the projection $\tau_1 : W \to E$ is generically one-to-one, hence birational, and therefore an isomorphism since $E \cong \P^1$. Composing with the projection $\tau_2 : W \to Z'$, we obtain
    \[
        \tau = \tau_2 \circ \tau_1^{-1} : E \to Z',
    \]
    a morphism of degree $D'$, giving the desired parametrization.

    We now show that for every $C$ with $[C] \in Z'$, we have $[F(C)] \in Z'$, after replace $F$ with some iterate. View $Z'$ as a curve
    \[
        Z'_t \subseteq \P^2_{\C(E)} \cong \P^2_{\C(t)},
    \]
    where $\C(E) \cong \C(t)$ is the function field of $E \cong \P^1$. The map $F$ induces a morphism
    \[
        \tilde{F} : \P^2_{\C(t)} \to \P^2_{\C(t)},
    \]
    and since $\deg(F(C')) = \deg(C')$ for infinitely many $C' \in Z'$, we have
    \[
        \deg(\tilde{F}(Z'_t)) = \deg(Z'_t).
    \]
    For any $C_1 = Z'_{t_1} \in \mathcal{C}'$, where $t_1$ is a non-superattracting periodic point of $F|_E$, Lemma~\ref{lem: transvers-non-super-E} and the construction of $Z'$ ensures that
    \[
        F(C_1) = Z'_{f(t_1)},
    \]
    where $f = F|_E$. As such points $t_1$ form a Zariski-dense subset of $E \cong \P^1$, we obtain
    \[
        \tilde{F}(Z'_t) = Z'_{f(t)}.
    \]
    Thus, for every $C$ with $[C] \in Z'$, we have $[F(C)] \in Z'$.

    Finally, since $Z \subseteq Z'$ and $Z'$ is irreducible, we have $Z = Z'$. Therefore, for any $C$ with $[C] \in Z$, we have $[F(C)] \in Z$, and for a generic $C$, 
    \[
        \deg(F(C)) = \deg(C).
    \]
 \end{proof}

We now summarize the results obtained for the case where $F$ admits an infinite set of periodic curves.

\begin{prop}\label{prop: 1-D-cases-summary}
    Let $F$ be a regular polynomial endomorphism of degree $d > 1$ admitting an infinite set $\mathcal{C}$ of periodic curves of degree at most $D$. Then there exists a parametrization 
    \[
    \tau : \P^1 \longrightarrow Z,
    \]
    where $Z \subseteq \mathcal{M}_D$ is the Zariski closure, in the space of effective divisors of degree at most $D$, of the effective divisors corresponding to an infinite subset of $\mathcal{C}$. Moreover, after replacing $F$ by a suitable iterate, for any curve $C$ and for a generic curve $C'$ with $[C], [C'] \in Z$, we have
    \[
    [F(C)] \in Z \quad \text{and} \quad \deg(F(C')) = \deg(C').
    \]
\end{prop}

\begin{proof}
    By Lemma~\ref{lem: one-point-or-web}, either there exists an infinite subset of $\mathcal{C}$ consisting of curves all passing through a single point $p \in H_\infty$, or the statement already holds for an infinite subset of $\mathcal{C}$, in which case we are done. 
    
    
    The first case is covered by Proposition~\ref{prop: moduli-DMM-1-D}.
\end{proof}

We are now ready to prove Theorem~\ref{thm: main-1}.

\begin{proof}
    Without loss of generality, we may assume that $Z$ is irreducible. We first show that for any curve $C$ with $[C] \in Z$, one has $[F(C)] \in Z$. 
    
    The pushforward map 
    \[
    F_* : \mathcal{M}_d \longrightarrow \mathcal{M}_{\deg(F)d}
    \]
    is well defined. Define 
    \[
    \psi : \mathcal{M}_d \longrightarrow \mathcal{M}_{\deg(F)d}, \quad \psi(A) = \deg(F)\,A,
    \]
    which is an isomorphism onto its image. Let $\mathcal{C}$ denote the Zariski dense set of periodic curves of $F$ whose corresponding effective divisors lie in $Z$. By Proposition~\ref{prop: 1-D-cases-summary}, after replacing $F$ by an iterate, there exists an infinite subset $\mathcal{C}' \subseteq \mathcal{C}$ whose Zariski closure $Z'$ satisfies $[F(C')] \in Z'$ for all $C' \in \mathcal{C}'$. 
    
    If $\mathcal{C}'$ is not Zariski dense in $Z$, then $\mathcal{C} \setminus \mathcal{C}'$ must still be infinite. We may then apply Proposition~\ref{prop: 1-D-cases-summary} again to $\mathcal{C} \setminus \mathcal{C}'$ and enlarge $\mathcal{C}'$ and $Z'$, where $Z'$ becomes a union of projective lines contained in the Zariski closure of $\mathcal{C}'$. By Zorn’s Lemma, there exists a maximal subset $\mathcal{C}'' \subseteq \mathcal{C}$ whose Zariski closure satisfies this invariance property. Our argument implies that $|\mathcal{C} \setminus \mathcal{C}''| < \infty$, since otherwise $\mathcal{C}''$ could be enlarged, contradicting maximality. We rename $\mathcal{C}''$ as $\mathcal{C}'$, noting that it remains Zariski dense in $Z$.
    
    Define 
    \[
    G \coloneqq \psi^{-1} \circ F_*.
    \]
    This map is well defined on the set of effective divisors corresponding to $\mathcal{C}'$, a Zariski dense subset of $Z$, and its image also lies in $Z$. Since $\psi(\mathcal{M}_d)$ is closed in $\mathcal{M}_{\deg(F)d}$ and contains $\deg(F)\mathcal{C}'$, the composition $\psi^{-1} \circ F_*$ extends to an endomorphism on a closed subset of $\mathcal{M}_d$ containing $\mathcal{C}'$. Hence $G$ extends to an endomorphism of $Z$, and we have $G(Z) = Z$. Consequently, for any curve $C$ with $[C] \in Z$, one has $[F(C)] \in Z$.
    
    Let $\tilde{Z}$ denote the curve in $\P^2_{\C(Z)}$ whose fibers over points of $Z$ correspond to the curves in $Z$, where $\C(Z)$ is the function field of $Z$. Let $\tilde{F} : \P^2_{\C(Z)} \to \P^2_{\C(Z)}$ be the induced map of $F$. The discussion above shows that $\tilde{F}(\tilde{Z})$ is a curve in $\P^2_{\C(Z)}$ containing a Zariski dense set of fibers corresponding to curves in $\mathcal{C}'$. Hence
    \[
    \deg(\tilde{F}(\tilde{Z})) = \deg(\tilde{Z}).
    \]
    Therefore, for a generic curve $C$ with $[C] \in Z$, we have
    \[
    \deg(F(C)) = \deg(C),
    \]
    completing the proof.
\end{proof}

In fact, a more general version of Theorem \ref{thm: main-1} is predicted by the Dynamical Manin--Mumford Conjecture:

\begin{thm}[Theorem \ref{thm: DMM-implies-general-main}]
    Let $F$ be a surjective endomorphism of $\P^K$ of degree $>1$. Let 
    \[
    Z \subseteq Ch_{D,d}(\P^K)
    \]
    be a subvariety of the Chow variety parameterizing algebraic cycles of $\P^K$ of dimension $D \leq K-1$ and degree $d$. Suppose $Z$ contains a Zariski dense set of algebraic cycles corresponding to irreducible periodic subvarieties of $\P^K$ under $F$ of degree $d$. 
    
    Assuming $(F, Z)$ is of general type and Conjecture \ref{conj: DMM}, then for every subvariety $C \subseteq \P^K$ and a generic subvariety $C' \subseteq \P^K$ with $[C], [C'] \in Z$, we have
    \[
    [F(C)] \in Z \quad \text{and} \quad \deg(F(C')) = \deg(C'),
    \]
    after possibly replacing $F$ by some iterate.
\end{thm}

\begin{proof}
    Without loss of generality, we may assume $Z$ is irreducible, since the argument can be applied to each irreducible component separately. Set $N \coloneqq \dim(Z) + 1$. Inspired by the fiber products constructions used in \cite{DM24}, we consider the subvariety
    \[
    Y \coloneqq \overline{\bigcup_{z \in Z} \{(x_1, \dots, x_N) : x_i \in C \subseteq \P^K,\, [C] = z,\, i = 1, \dots, N\}} \subseteq (\P^K)^N.
    \]
    Since $Z$ contains a Zariski dense set of periodic subvarieties under $F$, the set of preperiodic points of $F^{\times N}$ in $Y$ is Zariski dense. By Conjecture \ref{conj: DMM}, there exists a subvariety $W$ invariant under $(F^{\times N})^l$ for some $l > 0$, containing $Y$ and a polarizable 
endomorphism $G$ on $W$ such that 
\[
G \circ (F^{\times N})^l = (F^{\times N})^l \circ G
\quad \text{on } W.
\]
 Moreover, $Y$ is preperiodic under $G$.  

Since $Y$ contains every $C^N$ whose class $[C]$ lies in $Z$, and since $G$ becomes a split 
morphism after replacing it by a suitable iterate, we have 
\[
G_i \circ F^l = F^l \circ G_i
\quad \text{on } W_i,
\]
for each $i \in \{1, 2, \dots, N\}$, where $W_i$ denotes the projection of $W$ onto the $i$-th 
factor of $(\P^M)^N$, and $W_i$ contains $C$ for every $[C] \in Z$.  

By our assumption that $(F, Z)$ is of general type, it follows that each $G_i$ is an iterate of $F$, 
and therefore $G$ itself is an iterate of $F^{\times N}$. Hence, $Y$ is preperiodic under 
$F^{\times N}$.

    Then, there exist $m,n \in \Z^+$ such that
    \[
    (F^{\times N})^{m}(Y) \subseteq (F^{\times N})^{n}(Y).
    \]
    Assuming $n > m$, it follows that 
    \[
    (F^{\times N})^{m + (n-m)k}(Y) \subseteq (F^{\times N})^m(Y)
    \quad \text{for all } k \in \Z^+.
    \]
    Let $Z_0 \subseteq Z$ denote the set of cycles corresponding to the Zariski dense set of periodic subvarieties in $Z$. Then there exists some $j \in \{1, \dots, n-m\}$ such that
    \[
    \bigcup_{z \in Z'_0} \{(x_1, \dots, x_N) \in C^N : [C] = z\}
    \subseteq (F^{\times N})^{m+j}(Y),
    \]
    where $Z'_0 \subseteq Z_0$ is a Zariski dense subset of $Z$. Note that
    \[
    \overline{\bigcup_{z \in Z'_0} \{(x_1, \dots, x_N) \in C^N : [C] = z\}} = Y,
    \]
    hence
    \[
    Y = (F^{\times N})^{m+j}(Y),
    \]
    so $Y$ is periodic under $F^{\times N}$.

    Replacing $F$ by $F^{m+j}$, we have
    \[
    Y = \bigcup_{z \in Z} \{(x_1, \dots, x_N) \in C^N : [C] = z\}.
    \]
    Indeed, this follows because the union above equals
    \[
    \pi(\mathcal{Z}),
    \]
    where $\mathcal{Z}$ is the closed subvariety
    \[
    \mathcal{Z} \coloneqq \{(z, x_1, \dots, x_N) : z \in Z,\ (x_1, \dots, x_N) \in C^N,\ [C] = z\} \subseteq Z \times (\P^K)^N,
    \]
    and $\pi : Z \times (\P^K)^N \to (\P^K)^N$ is the natural projection, which is a closed map.

    Now let $C \subseteq \P^K$ be any subvariety with $[C] \in Z$. We claim that
    \[
    (F(C))^N \subseteq \bigcup_{z \in Z} \{(x_1, \dots, x_N) \in C'^N : [C'] = z\}
    \]
    implies $[F(C)] \in Z$.

    We prove this claim by induction on $N \ge 1$.  
    For the base case $N=1$, we have $\dim(Z)=0$, and $Y$ is a finite union of subvarieties of $\P^K$ parameterized by $Z$. In this case, $F(C) \subseteq Y$ clearly implies $[F(C)] \in Z$.

    Assume the statement holds for $N \le M$ for some $M \ge 1$, and consider $N = M+1$. There exists a point $p \in F(C)$ such that the subvariety
    \[
    Z_1 \coloneqq \{ z \in Z : p \in \mathrm{Supp}(z)\}
    \]
    has $\dim(Z_1) < N-1$. Otherwise, if every $C'$ with $[C'] \in Z$ contains $F(C)$, then $\dim(Z) = 0$, contradicting our assumption that $Z$ parametrizes a Zariski dense family of irreducible subvarieties.

    Hence,
    \[
    \{p\} \times (F(C))^{N-1} \subseteq \bigcup_{z \in Z_1} \{(p, x_2, \dots, x_N) : (x_2, \dots, x_N) \in C'^{N-1}, [C'] = z\},
    \]
    which implies
    \[
    (F(C))^{N-1} \subseteq \bigcup_{z \in Z_1} \{(x_2, \dots, x_N) : (x_2, \dots, x_N) \in C'^{N-1}, [C'] = z\}.
    \]
    By the induction hypothesis, we conclude that $[F(C)] \in Z$, proving the claim.

    Since $(F(C))^N \subseteq Y$, the claim yields $[F(C)] \in Z$. Let $\C(Z)$ be the function field of $Z$, and let
    \[
    \tilde{Z} \subseteq \P^K_{\overline{\C(Z)}}
    \]
    denote the subvariety corresponding to the generic fiber, so that for any $z \in Z$, we have $[\tilde{Z}_z] = z$. Denote by
    \[
    \tilde{F} : \P^K_{\overline{\C(Z)}} \to \P^K_{\overline{\C(Z)}}
    \]
    the map induced by $F$ under base change. Then
    \[
    [\tilde{F}(\tilde{Z}_z)] = [F(\tilde{Z}_z)] \in Z
    \quad \text{for all } z \in Z.
    \]
    In particular, this implies $\deg(\tilde{F}(\tilde{Z})) = \deg(\tilde{Z})$, and hence for a generic $C' \subseteq \P^K$ with $[C'] \in Z$, we have
    \[
    \deg(F(C')) = \deg(C').
    \]
\end{proof}

 \section{Implication Towards the Conjectures \ref{conj: DM24-conj-1,1}}\label{sect: impl-RDMM}

 In this section, we explore how our main result relates to the conjectures proposed in \cite{DM24}. 
When we restrict to the case where the family 
\[
\Phi : S \times \P^2 \longrightarrow S \times \P^2
\]
is constant with respect to the parameter space $S$, we show that our main theorem implies a weaker form of Conjecture~\ref{conj: DM24-conj-1,1}, stated below as Corollary~\ref{cor: conj-1.1-DM}.

For the convenience of the reader, we restate Corollary~\ref{cor: conj-1.1-DM} here.
 
 \begin{cor}[Corollary \ref{cor: conj-1.1-DM}]
      Let $S$ be a smooth and irreducible quasi-projective variety defined over $\C$. Let $\Phi : S \times \P^2 \to S \times \P^2 $ be a constant family of endomorphisms such that for every point $(s, p) \in S \times \P^2$
     $$ \Phi(s,p) = (s, F(p))$$
     for a regular polynomial endomorphism $F$ of degree $>1$ does not depend on $s$. Let $\mathcal{X} \subseteq S \times \P^2$ be an irreducible hypersurface that is flat over $S$ and projects dominantly to $S$. Suppose there exists a Zariski dense set of $s \in S$ such that $\mathcal{X}_s$, the fiber of $\mathcal{X}$ at $s$, is a periodic curve under $F$. Then for all $N \in \Z^+$, we have 
     $$ \hat{T}_{\Phi^{\times N}}^{\wedge r_{\Phi^{\times N}, \mathcal{X}^N}} \wedge [\mathcal{X}^N] \neq 0,$$
     and $r_{\Phi^{\times N}, \mathcal{X}^N} \leq \min\{ 2N , N + \dim S\}$.
 \end{cor}
 \begin{proof}

    There exists an irreducible subvariety 
    \[
    Z \subseteq \mathcal{M}_d,
    \]
    where $\mathcal{M}_d$ denotes the space of effective divisors of degree at most $d > 0$ in $\P^2$, such that the induced map by the family $\mathcal{X}$
    \[
    \sigma : S \to Z
    \]
    is dominant. 
    When no confusion arises, we will also regard $\mathcal{M}_d$ as the space of (possibly reducible) curves in $\P^2$, and similarly for $Z \subseteq \mathcal{M}_d$. 
    By assumption, there exists a Zariski dense set of $z \in Z$ such that $z$ is a periodic curve under $F$. 
    Then Theorem~\ref{thm: main-1} implies that, after replacing $F$ by some iterate, $F(z) \in Z$ for all $z \in Z$. 
    
    Consider the subvariety
    \[
    Y \coloneqq \overline{\{(p_1, \dots, p_N) \in (\P^2)^N : p_1, \dots, p_N \in \mathcal{X}_s \text{ for some } s \in S\}},
    \]
    which is the image of the projection of the fiber product $\mathcal{X}^N$ onto $(\P^2)^N$. 
    Then
    \[
    \dim(Y) = r \leq \min\{2N,\, \dim(\mathcal{X}^N)\} \leq \min\{2N,\, N + \dim S\}.
    \]
    We claim that $Y$ is invariant under $F^{\times N}$. 
    Indeed, for any $(p_1, \dots, p_N) \in (\mathcal{X}_s)^N \subseteq (\P^2)^N$, we have
    \[
    (F(p_1), \dots, F(p_N)) \in F(\mathcal{X}_s)^N = F(\sigma(s))^N = z^N
    \]
    for some $z \in Z$. 
    Since $\sigma : S \to Z$ is dominant and $Y$ is Zariski closed, we have $z^N \subseteq Y$, hence $F^{\times N}(Y) \subseteq Y$.

    Applying the projection formula for 
    \[
    \tau_N : S \times (\P^2)^N \to (\P^2)^N,
    \]
    we obtain
    \[
    \hat{T}_{\Phi^{\times N}}^{\wedge r} \wedge [\mathcal{X}^N]
    = \tau_N^*(\hat{T}_{F^{\times N}}^{\wedge r}) \wedge [\mathcal{X}^N]
    = \hat{T}_{F^{\times N}}^{\wedge r} \wedge (\tau_N)_*[\mathcal{X}^N]
    \geq \hat{T}_{F^{\times N}}^{\wedge r} \wedge [Y].
    \]
    
    Let $T = \left(\sum_{i=1}^N \pi_i^*\omega\right)^{\wedge r}$ be an $r$-current on $(\P^2)^N$, where $\pi_i$ denotes projection onto the $i$-th factor and $\omega$ is the Fubini–Study form on $\P^2$. 
    Then
    \[
    \int_{(\P^2)^N} T \wedge [Y] > 0.
    \]
    Let $d = \deg(F^{\times N})$. 
    Observe that
    \begin{align*}
        \int_{(\P^2)^N} d^{-rn} (F^{\times N})^{n*}T \wedge [Y] 
        &= \int_{(\P^2)^N} d^{-rn} T \wedge (F^{\times N})^n_*[Y] \\
        &= \int_{(\P^2)^N} d^{-rn} T \wedge d^{rn}[Y] \\
        &= \int_{(\P^2)^N} T \wedge [Y].
    \end{align*}
    Moreover, by the local uniform convergence of the potentials of $(F^{\times N})^{n*}T$ to those of $\hat{T}_{F^{\times N}}^{\wedge r}$ (see \cite[Chapter~III, Corollary~3.6]{Dem}), we have
    \[
    d^{-rn}(F^{\times N})^{n*}T \wedge [Y] \longrightarrow \hat{T}_{F^{\times N}}^{\wedge r} \wedge [Y]
    \quad \text{as } n \to \infty.
    \]
    Hence,
    \[
    \int_{(\P^2)^N} \hat{T}_{F^{\times N}}^{\wedge r} \wedge [Y] > 0,
    \]
    and therefore,
    \[
    \hat{T}_{F^{\times N}}^{\wedge r} \wedge [Y] > 0.
    \]

    Finally, by definition, we have $r_{\Phi^{\times N}, \mathcal{X}^N} \leq r$ since $S \times Y$ is $\Phi^{\times N}$-special. 
    Consequently,
    \[
    \hat{T}_{\Phi^{\times N}}^{\wedge r_{\Phi^{\times N}, \mathcal{X}^N}} \wedge [\mathcal{X}^N] \neq 0.
    \]

 \end{proof}

 \begin{rmk}\label{rmk: special-case-DMM}

     The proof above in fact shows that we have established a version of the Dynamical Manin–Mumford Conjecture for $F^{\times N}$ and $Y \subseteq (\P^2)^N$. 
However, our result holds under the stronger assumption that $Y$ arises from a family of curves containing a Zariski dense set of \emph{periodic curves}, rather than merely requiring that $Y$ contains a Zariski dense set of \emph{preperiodic points}.
 \end{rmk}
 \begin{cor}
    Under the same assumptions of Corollary \ref{cor: conj-1.1-DM}, for any positive integer $N > 0$, we have $\mathcal{X}^N $ has codimension $\leq \dim S$ in a $\Phi^{\times N}$-special subvariety in $S \times (\P^2)^N$.
 \end{cor}
 \begin{proof}

This follows directly from Corollary \ref{cor: conj-1.1-DM}, in the same way that \cite[Conjecture 1.1]{DM24} implies \cite[Conjecture 1.2]{DM24}. 
Since we have 
\[
r_{\Phi^{\times N}, \mathcal{X}^N} \leq \min\{2N,\, N + \dim S\}
\quad \text{and} \quad 
\dim \mathcal{X}^N = N + \dim S,
\]
it follows that there exists a $\Phi^{\times N}$-special subvariety 
\[
Y \subseteq S \times (\P^2)^N
\]
of dimension at most 
\[
\min\{2N + \dim S,\, N + 2\dim S\}
\]
containing $\mathcal{X}^N$. 
Hence, $\mathcal{X}^N$ has codimension at most $\dim S$ within $Y$.
 
 \end{proof}

The corollary \ref{cor: conj-1.1-DM} establishes a special case of the following conjecture. 

\begin{conj}\label{conj-const-1.2}
Let $S$ be a smooth and irreducible quasi-projective variety defined over $\C$, 
and let $K$ be a positive integer. 
Let $\Phi : S \times \P^K \to S \times \P^K$ be a family of endomorphisms such that for every point $(s, p) \in S \times \P^K$,
\[
\Phi(s,p) = (s, F_s(p)),
\]
where $F_s$ is an endomorphism of $\P^K$ of degree greater than $1$. 
Let $\mathcal{X} \subseteq S \times \P^K$ be an irreducible subvariety that projects dominantly onto $S$, 
is flat over $S$, and such that $\mathcal{X}_s$ is a preperiodic subvariety under $F_s$ for a Zariski dense set of $s \in S$. 
Then for any positive integer $N$, we have
\[
\hat{T}^{\,r_{\Phi^{\times N}, \mathcal{X}^N}}_{\Phi^{\times N}} \wedge [\mathcal{X}^N] \neq 0.
\]
\end{conj}

We note explicitly that this conjecture is a direct consequence of Conjecture \ref{conj: DM24-conj-1,1} 
by embedding $(\P^K)^N$ into some projective space $\P^M$ via the Segre embedding 
and viewing $\mathcal{X}^N$ as a subvariety of $\P^M$. 
This makes clear how our result connects to Conjecture \ref{conj: DM24-conj-1,1} proposed in \cite{DM24}.

 \begin{lem}
     Conjecture \ref{conj: DM24-conj-1,1} implies Conjecture \ref{conj-const-1.2}
 \end{lem}
 \begin{proof}

   Notice that we can view $\mathcal{X}^N$ as a subvariety of $S \times \P^M$ for some positive integer $M$ via the Segre embedding. 
Under this embedding, $\Phi^{\times N}$ extends naturally to an endomorphism on $S \times \P^M$, 
since each $F_s$ is polarizable for $s \in S$. 
Let 
\[
\iota : S \times (\P^K)^N \longrightarrow S \times \P^M
\]
denote this embedding map, and let 
\[
\Psi : S \times \P^M \longrightarrow S \times \P^M
\]
be the extended endomorphism induced by $\Phi^{\times N}$. 
Set $Y = \iota(\mathcal{X}^N)$. 
Then
\[
\hat{T}^{\,r_{\Psi, Y}}_{\Psi} \wedge [Y] \neq 0
\quad \Longrightarrow \quad
\hat{T}^{\,r_{\Psi, Y}}_{\Phi^{\times N}} \wedge [\mathcal{X}^N] \neq 0.
\]

Observe that $S \times \iota((\P^K)^N) \subseteq S \times \P^M$ is clearly $\Psi$-special and contains $Y$, 
so that $r_{\Psi, Y} = r_{\Phi^{\times N}, \mathcal{X}^N}$. 
Moreover, the assumption that for a Zariski dense set of $s \in S$, 
the fiber $\mathcal{X}_s$ is preperiodic under $F_s$ implies that 
$\Prep(\Phi) \cap \mathcal{X}_s$ is Zariski dense in $\mathcal{X}_s$ for such $s$. 
Consequently, for those $s$, the fiber $(\mathcal{X}^N)_s$ contains a Zariski dense subset of $\Prep(\Phi^{\times N}_s)$, 
and hence $\Prep(\Phi^{\times N}) \cap \mathcal{X}^N$ is Zariski dense in $\mathcal{X}^N$. 
Therefore, $\iota(\Prep(\Phi^{\times N}) \cap \mathcal{X}^N)$ is Zariski dense in $Y$, 
and Conjecture~\ref{conj: DM24-conj-1,1} directly implies the desired statement.
 \end{proof}
\section{Degree Stabilization for Families of Endomorphims}\label{sect: degree-stabilization}
In this section, we collect some conditional results on degree stabilization allowing for a family of varying endomorphisms. Let's first show that Conjecture \ref{conj: DM24-conj-1,1} implies the degree stabilization phenomenon in this setting of a family of endomorphisms. 

\begin{prop}\label{prop: special-implies-deg}
     Let $S$ be a smooth and irreducible quasi-projective variety defined over $\C$ of positive dimension and $M$ be a positive integer. Let $\Phi : S \times \P^M \to S \times \P^M$ be a family of endomorphisms such that for every point $(s,p) \in S \times \P^M$ 
     $$\Phi(s,p) = (s, F_s(p))$$
     for an endomorphism $F_s$ of degree $>1$. Let $\mathcal{X} \subseteq S \times \P^M$ be an irreducible subvariety of dimension $>\dim S$ that projects dominantly to $S$ and is flat over $S$.

     If there exists a positive integer $N $ and a proper irreducible $\Phi^{\times N}$-preperiodic subvariety $\mathcal{Y} \subseteq S \times (\P^M)^N$ containing $(\Phi^{\times N})^l(\mathcal{X}^N)$ with some non-negative integer $l$ as a codimension $\leq \min \{N-1, \dim S\}$ subvareity, then there exists positive integers $n,m$ and an infinite set of positive integers $I$ such that for any pair of integers $k_1, k_2 \in I$ we have that for a generic point $s \in S$, $$\deg(F^{m + nk_1}_s(\mathcal{X}_s)) = \deg(F^{m + nk_2}_s(\mathcal{X}_s)).$$
 \end{prop}
 \begin{proof}

     By our assumption , we know that there exists positive integers $n,m$ such that $(\Phi^{\times N})^m(\mathcal{Y}) = (\Phi^{\times N})^{m+nk}(\mathcal{Y})$ for all $k \in \N$.

     Let's prove it inductively on $N$. The base case is $N =1$. In this case we have $\Phi^l(\mathcal{X})$ itself is $\Phi$-special and, by our assumption on $\Phi$, this implies $\mathcal{X}$ is preperiodic under $\Phi$. Therefore, for every $s \in S$, we have $$F_s^m(\mathcal{X}_s) = F_s^{m+nk}(\mathcal{X}_s)$$
     and in particular
     $$ \deg(F^m_s(\mathcal{X}_s)) = \deg(F^{m + nk}_s(\mathcal{X}_s)),$$
     for all positive integer $k$.

     Now, suppose $N \geq 2$ and assume that we have the statement proved for any $N' \leq N-1$. If the projection of $\mathcal{Y}$ to $S \times (\P^M)^{N-1}$, denote as $\pi(\mathcal{Y})$, contains $\mathcal{X}^{N-1}$ as a codimension $\leq \min\{N-2, \dim S\}$ subvariety, then we apply the induction hypothesis to conclude the proof.

     Otherwise, we will have $$ \dim(\mathcal{X}^{N-1}) + \min\{N-1, \dim S\} \geq \dim(\pi(\mathcal{Y})) > \dim(\mathcal{X}^{N-1}) + \min\{N-2, \dim S\},$$
     since $\dim(\mathcal{Y}) \leq \dim(\mathcal{X}^N) + \min\{N-1, \dim S\}$. By \cite[Chapter \RNum{3}, Theorem 9.9 and Corollary 9.10]{Har77}, there exists a non-empty Zariski open subset $U$ in $\pi(\mathcal{Y})$ such that 
      $\deg(\mathcal{Y}_p)$
     is constant with $p$ varies in $U$, where $\mathcal{Y}_p$ is the fiber of $\mathcal{Y}$ over $p$. 
     
     Suppose there exists positive integers $k_0,k_1$ such that for every $k > k_0$, we have that $(\Phi^{\times (N-1)})^{m + nk_1k}(\mathcal{X}^{N-1}) \subseteq U^c$. Then there exists a closed subset $$V \coloneqq \overline{\bigcup_{k \geq k_0}(\Phi^{\times (N-1)})^{m + nk_1k}(\mathcal{X}^{N-1})} \subseteq U^c $$
     which is $\Phi^{\times(N-1)}$ preperiodic and contains $(\Phi^{\times (N-1)})^{l'}(\mathcal{X}^{N-1})$
     for some positive integer $l'$. Then the induction hypothesis will conclude the proof as $\dim(V) < \dim(\pi(\mathcal{Y}))$.

     Suppose such a pair of integers does not exist. This in particular implies that there is an infinite set of positive integers $I$ such that for any $j \in I$  $(\Phi^{\times (N-1)})^{m + nj}(\mathcal{X}^{ N-1})$ is not contained in $U^c$. Then for any pair of integers $k_3, k_4 \in I$, we have
     $$ (\Phi^{\times (N-1)})^{m + nk_i}(\mathcal{X}^{N-1}) $$
     is not contained in $U^c$, for $i \in \{3,4\}$. Since $U \cap (\Phi^{\times (N-1)})^{m + nk_i}(\mathcal{X}^{N-1})$ is a non-empty open subset in $(\Phi^{\times (N-1)})^{m + nk_i}(\mathcal{X}^{N-1})$ for $i\in \{3,4\}$, this implies that for a generic $s_0 \in S$, there exists $$p_{k_i} \in U \cap (\Phi^{\times (N-1)})^{m + nk_i}(\mathcal{X}^{N-1})$$
     such that $\pi_0(p_{k_i}) = s_0 \in S$ and $$F_{s_0}^{m+nk_i}(\mathcal{X}_{s_0}) \subseteq \mathcal{Y}_{p_{k_i}},$$
     with $i \in \{3,4\}$. Therefore, since $\dim(F_{s_0}^{m+nk_i}(\mathcal{X}_{s_0})) = \dim(\mathcal{Y}_{p_{k_i}})$, for $i\in \{3,4\}$, $\deg(\mathcal{Y}_{p_{k_3}}) = \deg(\mathcal{Y}_{p_{k_4}})$, we have 
$$\deg(\mathcal{Y}_{p_{k_3}}) \geq \deg(F_{s_0}^{m+nk_3}(\mathcal{X}_{s_0})) ,$$
$$\deg(\mathcal{Y}_{p_{k_4}}) \geq  \deg(F_{s_0}^{m+nk_4}(\mathcal{X}_{s_0})) .$$
This implies that for all pair of $k_3, k_4 \in I$, we have 
$$\deg(F_{s_0}^{m+nk_3}(\mathcal{X}_{s_0})),  \deg(F_{s_0}^{m+nk_4}(\mathcal{X}_{s_0})) \leq R$$
for a generic $s_0 \in S$ and a positive integer $R$ only depending on $U$ and $\mathcal{Y}$. Then, by the pigeonhole princple, there exists an infinite subset $I' \subseteq I$ such that for all pairs of $k_3, k_4 \in I'$, we have

     $$ \deg(F^{m+nk_3}_{s_0}(\mathcal{X}_{s_0})) = \deg(F^{m+ nk_4}_{s_0}(\mathcal{X}_{s_0})) \leq R.$$

 \end{proof}

\begin{prop}[Proposition \ref{conj: cor-degree-stable-family}]
Under the same assumptions as Proposition \ref{prop: special-implies-deg}, 
suppose that for a Zariski dense subset of $s \in S$, the fiber $\mathcal{X}_s$ is preperiodic under $\Phi_s$. Assume $(\Phi, \mathcal{X})$ is of general type.
If Conjecture \ref{conj: DM24-conj-1,1} holds, then there exist positive integers $n, m$ 
and an infinite set of positive integers $I$ such that for any pair of integers $k_1, k_2 \in I$, 
we have, for a generic point $s \in S$,
\[
\deg(F^{m + nk_1}_s(\mathcal{X}_s)) = \deg(F^{m + nk_2}_s(\mathcal{X}_s)).
\]
\end{prop}

\begin{proof}
Note
\[
\hat{T}^{\,r_{\Phi^{\times N}, \mathcal{X}^N}}_{\Phi^{\times N}} \wedge [\mathcal{X}^N] \neq 0,
\]
implies $\dim (\mathcal{X}^N) \geq r_{\Phi^{\times N}, \mathcal{X}^N}$, 
which implies the existence of a $\Phi^{\times N}$-special subvariety of dimension at most 
$\dim(\mathcal{X}^N) + \dim S$. 
Since we have already shown that Conjecture \ref{conj: DM24-conj-1,1} implies Conjecture \ref{conj-const-1.2}, our assumption implies that 
we can find a $N > \dim S$ such that there exists a proper $\Phi^{\times N}$-special subvariety 
$\mathcal{Y} \subseteq S \times (\P^M)^N$ of dimension at most
\[
\dim(\mathcal{X}^N) + \dim S = N + 2\dim S < \dim(S \times (\P^M)^N),
\]
containing $\mathcal{X}^N$. 
Since $\mathcal{X}^N$ is irreducible, there exists an irreducible component of $\mathcal{Y}$ containing $\mathcal{X}^N$; 
by abuse of notation, we denote this component again by $\mathcal{Y}$.

Note that $\mathcal{Y}$ being $\Phi^{\times N}$\textnormal{-special} means that there exists 
a subvariety $\mathbf{Z} \subseteq (\mathbf{P}^M)^N$ containing $\mathbf{Y}$ which is invariant 
under a polarizable endomorphism $\mathbf{\Psi}$ on $(\mathbf{P}^M)^N$ and also under 
$(\mathbf{\Phi}^{\times N})^n$ for some $n > 0$, such that 
\[
\mathbf{\Psi} \circ (\mathbf{\Phi}^{\times N})^n 
= 
(\mathbf{\Phi}^{\times N})^n \circ \mathbf{\Psi}
\quad \text{on } \mathbf{Z}.
\]
Moreover, $\mathbf{Y}$ is preperiodic under $\mathbf{\Psi}$.

Since $\mathbf{\Psi}$ is polarizable, by replacing it with a suitable iterate, we may assume that 
$\mathbf{\Psi}$ is a split morphism.  
Because $\mathcal{Y}$ contains $\mathcal{X}^N$ and $\mathbf{Z}$ contains $\mathbf{Y}$, 
for each $i \in \{1, 2, \dots, N\}$ our assumptions imply that 
\[
\mathbf{\Psi}_i \circ \mathbf{\Phi}^n 
= 
\mathbf{\Phi}^n \circ \mathbf{\Psi}_i
\quad \text{on } \mathbf{Z}_i,
\]
where $\mathbf{Z}_i$ is the projection of $\mathbf{Z}$ to the $i$-th factor of 
$(\mathbf{P}^M)^N$ and contains $\mathbf{X}$, and $\mathbf{\Psi}_i$ denotes the corresponding 
projection of $\mathbf{\Psi}$ to that factor.  

By the assumption that $(\Phi, \mathcal{X})$ is of general type, 
it follows that $\mathbf{\Psi}$ must be an iterate of $\mathbf{\Phi}^{\times N}$. 
Consequently, $\mathbf{Y}$ is preperiodic under $\mathbf{\Phi}^{\times N}$, 
and hence $\mathcal{Y}$ is preperiodic under $\Phi^{\times N}$.



Then Proposition \ref{prop: special-implies-deg} applies, yielding the existence of positive integers 
$n, m$ and an infinite set of positive integers $I$ such that for any pair $k_1, k_2 \in I$, 
we have, for a generic point $s \in S$,
\[
\deg(F^{m + nk_1}_s(\mathcal{X}_s)) = \deg(F^{m + nk_2}_s(\mathcal{X}_s)).
\]
\end{proof}

\subsection{One-Dimensional Families of Polynomial Endomorphisms on $\P^2$ }
In this subsection, we verify the degree stabilization for one-dimensional families of regular polynomial endomorphisms under some ramification restriction on the line at infinity.

 Suppose $S = \A^1$, and we view 
\[
\Phi : \A^1 \times \P^2 \longrightarrow \A^1 \times \P^2
\]
as a family $F_t : \P^2_{\overline{\C(t)}} \to \P^2_{\overline{\C(t)}}$ parametrized by $t \in \P^1$. 
Then, the technique used in Section 3 can be applied to prove degree stabilization when, for each $a_i(t) \in \mathcal{X}_t \cap H_\infty$, the pair $(F_t|_{H_\infty}, a_i(t))$ does not share its forward orbit with a marked critical point.

\begin{thm}\label{thm: line-at-infinty-poly}
Let $F_t$ be a family of regular polynomial endomorphisms of $\P^2$ of degree $>1$, parametrized by $t \in \P^1$, and let $\mathcal{X}_t \subset \P^2$ be a family of curves. Suppose there are infinitely many $t_0 \in \P^1$ such that $\mathcal{X}_{t_0}$ is periodic under $F_{t_0}$, and set $f_t \coloneqq F_t|_{H_\infty}$ satisfying:
\begin{enumerate}
    \item $f_t$ is a polynomial of degree $>1$ over $\overline{\C(t)}$;
    \item For any $a_i(t) \in \mathcal{X}_t \cap H_\infty$, there do not exist $n \in \N$ and a polynomial $h_t$ sharing the same Julia set with $f_t$ such that
    \[
    f_t^n(a_i(t)) = h_t(c(t)),
    \]
    where $c(t)$ is a marked critical point of $h_t$, i.e., $(h_t)'(c(t)) \equiv 0$.
\end{enumerate}
Then, for a generic $t_0 \in \P^1$, we have
\[
\deg(\mathcal{X}_{t_0}) = \deg(F_{t_0}^k(\mathcal{X}_{t_0}))
\]
for all positive integers $k$.
\end{thm}

\begin{proof}
We first show that under the above assumptions, for any $a(t) \in \mathcal{X}_t \cap H_\infty$, there are only finitely many $t_1 \in \A^1$ such that $a(t_1)$ lies in a superattracting periodic cycle of $f_{t_1}$. 

Suppose, for contradiction, that there are infinitely many $t_1 \in \A^1$ and some $a(t) \in \mathcal{X}_t \cap H_\infty$ such that $a(t_1)$ is superattracting periodic under $f_{t_1}$. 
Since $f_t$ has bounded degree, it has finitely many marked critical points. By the pigeonhole principle, there exists a marked critical point $c(t)$ such that
\[
\{ t_1 \in \P^1 : c(t_1) \in \Per(f_{t_1}), \, a(t_1) = f_{t_1}^{\,n_{t_1}}(c(t_1)), \, n_{t_1} \in \N \}
\]
is infinite.

\textbf{Case 1:} $f_t$ is conjugate to a power map or Chebyshev polynomial. Then there exists $l_t \in \overline{\C(t)}[x]$ such that
\[
l_t \circ S_d \circ l_t^{-1} = f_t,
\]
where $S_d$ is either a power map or Chebyshev polynomial of degree $d = \deg(f_t)$. 
The marked critical points are
\[
\{ l_t(c_i) : c_i \text{ is a critical point of } S_d \}.
\]
Thus, for infinitely many $t_1$,
\[
a(t_1) = f_{t_1}^{\,n_{t_1}}(c(t_1)) = l_{t_1} \circ S_d^{\,n_{t_1}}(c),
\]
where $c$ is a critical point of $S_d$. Since all critical points of $S_d$ are preperiodic, this implies
\[
a(t_0) = l_{t_0} \circ S_d^n(c)
\]
for infinitely many $t_0 \in \A^1$, contradicting assumption (2).

\textbf{Case 2:} $f_t$ is non-special. Then there exists a primitive polynomial $g_t$ sharing the same Julia set as $f_t$ over $\overline{\C(t)}$ such that
\[
f_t = \sigma \circ g_t^k
\]
for some $k \in \N$ and $\sigma \in \Aut(\P^1)$ preserving $J(f_t)$ \cite{SS95}. 
Since $f_t$ and $g_t$ share the same preperiodic points, there exist infinitely many $t_1$ such that $a(t_1)$ and $c(t_1)$ are preperiodic under $g_{t_1}$. 
By \cite[Theorem D]{FG22}, there exist $\sigma_1 \in \Aut(\P^1)$ preserving $J(g_t)$ and integers $m,n > 0$ such that
\[
\sigma_1 \circ g_t^m(c(t)) = g_t^n(a(t)).
\]
Applying $\sigma' \circ g_t^{k'}$ on both sides of the equation above with sufficiently large $k'$ and some $\sigma' \in \Aut(\P^1)$ preserving $J(g_t)$ yields
\[
\sigma_2 \circ g_t^{l_1}(c(t)) = f_t^{\,l_2}(a(t)),
\]
with $l_1, l_2 > k$ and $\sigma_2 \in \Aut(\P^1)$ preserving $J(g_t)$. 
Since
\[
g_t^{l_1} = g_t^{l_1 - k} \circ \sigma^{-1} \circ f_t,
\]
$c(t)$ is a marked critical point of $g_t^{\,l_1}$. Let $h_t = \sigma_2 \circ g_t^{\,l_1}$, then $c(t)$ is also a marked critical point of $h_t$, contradicting assumption (2).

Hence, each $a(t) \in \mathcal{X}_t \cap H_\infty$ lies in a superattracting periodic orbit for only finitely many $t_1$. 
Since there are infinitely many $t_2 \in \A^1$ such that $\mathcal{X}_{t_2}$ is periodic under $F_{t_2}$, every $a(t_2) \in \mathcal{X}_{t_2} \cap H_\infty$ is periodic under $f_{t_2}$. 

By selecting an infinite subset of such $t_2$, we may assume each $a(t_3) \in \mathcal{X}_{t_3} \cap H_\infty$ is non-superattracting periodic. 
Then, by Lemma \ref{lem: transvers-non-super-E},
\[
\deg(\mathcal{X}_{t_3}) = |\mathcal{X}_{t_3} \cap H_\infty| = |F_{t_3}^k(\mathcal{X}_{t_3}) \cap H_\infty| = \deg(F_{t_3}^k(\mathcal{X}_{t_3}))
\]
for all $k \in \Z^+$. Since $F_t^k(\mathcal{X}_t)$ and $\mathcal{X}_t$ both form an algebraic family of curves in $\P^2$, it follows that for a generic $t_0 \in \P^1$,
\[
\deg(F_{t_0}^k(\mathcal{X}_{t_0})) = \deg(\mathcal{X}_{t_0}).
\]
\end{proof}

The following conjecture is a special case of the conjecture stated explicitly in \cite[Conjecture 6.1]{De16}:
\begin{conj}\label{Conj: De3-Conj-6.1}
    Let $f_t(x)$ be a non-isotrivial algebraic family of rational functions of degree $>1$ defined over $\C$ parametrized by $t$ varies in $\A^1_\C$ and $a(t)$, $b(t)$ are two marked points defined over $\overline{\C(t)}$. Suppose there are infinitely many $t_0 \in \A^1$ such that $a(t_0), b(t_0) \in \Prep(f_{t_0})$. Then $$\overline{\Orb_{(f_t,f_t)}((a(t), b(t)))} \subseteq (\P^1 \times \P^1)(\overline{\C(t)})$$ is proper Zariski closed.
\end{conj}
Assuming this Conjecture, we can strength Theorem \ref{thm: line-at-infinty-poly}:

\begin{thm}\label{thm: line-at-infinity-algebraic}
    Assume Conjecture~\ref{Conj: De3-Conj-6.1}. 
    Let $F_t$ be a family of regular polynomial endomorphisms on $\P^2$ defined over $\C$ of degree greater than $1$, parametrized by $t \in \A^1$, and let $\mathcal{X}_t$ be a family of curves in $\P^2$ defined over $\C$. 
    Suppose there are infinitely many $t_0 \in \P^1$ such that $\mathcal{X}_{t_0}$ is periodic under $F_{t_0}$. 
    Denote by $f_t \coloneqq F_t|_{H_\infty}$ the restriction of $F_t$ to the line at infinity. 
    Then, for a generic set of parameters $t_0 \in \A^1_\C$, we have 
    \[
        \deg(\mathcal{X}_{t_0}) = \deg(F_{t_0}^k(\mathcal{X}_{t_0}))
    \]
    for all positive integers $k$, unless there exist a point $a_i(t) \in \mathcal{X}_t \cap H_\infty$, a pair of integers $n,m \ge 0$, a rational map $X_t : \P^1 \to \P^1$, and two rational functions $g_t,h_t \in \overline{\C(t)}[x]$ such that 
    \begin{enumerate}
        \item $f_t \circ X_t = X_t \circ g_t$;
        \item $g_t$ commutes with $h_t$ and they share a common iterate;
        \item $f_t^n(a_i(t)) = X_t \circ h_t^m(c(t))$, where $c(t)$ is a marked critical point of $h_t^m$.
    \end{enumerate}
\end{thm}

\begin{proof}
    Suppose that for every $a(t) \in \mathcal{X}_t \cap H_\infty$, there exist only finitely many parameters $t_1 \in \A^1$ such that $a(t_1)$ lies in a superattracting periodic cycle of $f_{t_1}$. 
    Then, by the same argument as in the proof of Theorem~\ref{thm: line-at-infinty-poly}, we conclude that 
    \[
        \deg(F_{t_0}^k(\mathcal{X}_{t_0})) = \deg(\mathcal{X}_{t_0})
    \]
    for a generic set of $t_0 \in \A^1$ and all $k > 0$.

    Now suppose that $f_t$ lies entirely in the Latt\`es locus. 
    Since all periodic points of a Latt\`es map are repelling, it follows that for any $t_0 \in \A^1$ such that $\mathcal{X}_{t_0}$ is periodic under $F_{t_0}$, every $a(t_0) \in \mathcal{X}_{t_0} \cap H_\infty$ is a repelling periodic point. 
    Hence this case is complete.

    Assume next that $f_t$ is not contained in the Latt\`es locus. 
    Suppose there exist infinitely many parameters $t_1 \in \A^1$ and some $a(t) \in \mathcal{X}_t \cap H_\infty$ such that $a(t_1)$ is a superattracting periodic point of $f_{t_1}$. 
    By the pigeonhole principle, there exists a marked critical point $c(t)$ of $f_t$ such that 
    \[
        \{\, t_1 \in \A^1 : c(t_1) \in \Per(f_{t_1}),\ a(t_1) = f_{t_1}^{n_{t_1}}(c(t_1)),\ n_{t_1} \in \N \,\}
    \]
    is an infinite set. 
    Then, Conjecture~\ref{Conj: De3-Conj-6.1} implies that there exists an invariant curve 
    $V \subseteq \P^1 \times \P^1$ under $(f_t,f_t)$ such that $\Orb_{(f_t,f_t)}(a(t),c(t)) \subseteq V$.

    If $V$ is vertical or horizontal, then one of $a(t)$ or $c(t)$ is passive. 
    Suppose $c(t)$ is passive. 
    Then there exists $m \in \N$ such that $f_t^m(c(t)) = c(t)$. 
    Since $a(t_1)$ and $c(t_1)$ lie in the same periodic cycle for infinitely many $t_1$, we must also have $f_t^m(a(t)) = a(t)$ and $f_t^n(a(t)) = c(t)$ for some $n \in \N$. 
    Taking $X_t$ to be the identity map and $g_t = h_t = f_t$, the statement follows. 
    The same reasoning applies when $a(t)$ is passive.

    Hence, we may assume both $a(t)$ and $c(t)$ are active, so $V$ projects dominantly onto both coordinates. 
    If $f_t$ is conjugate to a power map or a Chebyshev map, then, as in the proof of Theorem~\ref{thm: line-at-infinty-poly}, we again have $a(t) = f_t^m(c(t))$ for some $m \ge 0$, and the theorem holds.

    Now assume that $f_t$ is non-special. 
    By \cite[Theorem~4.15]{Pa23}, there exist rational maps $X_t ,g_t \in \overline{\C(t)}(x)$, where $g_t$ is not a generalized Latt\`es map (in the sense of \cite{Pa20}), such that
    \[
        f_t \circ X_t = X_t \circ g_t,
    \]
    and an invariant curve $W$ under $(g_t,g_t)$ satisfying 
    \[
        V = (X_t,X_t)(W),
    \]
    after replacing $f_t$ and $g_t$ by suitable iterates if necessary. 
    Then, by \cite[Theorem~1.2]{Pa23}, there exist rational maps $U_1, U_2, V_1, V_2$ commuting with $g_t$ such that 
    \[
        U_1 \circ V_1 = V_1 \circ U_1 = g_t^l, \quad 
        U_2 \circ V_2 = V_2 \circ U_2 = g_t^l
    \]
    for some positive integer $l$, and $W$ is parametrized by $u \mapsto (U_1(u), U_2(u))$. 
    Consequently, there exists $u(t) \in \P^1_{\overline{\C(t)}}$ such that
    \[
        a(t) = X_t \circ U_1(u(t)), \quad 
        c(t) = X_t \circ U_2(u(t)).
    \]

    For any $n > 0$, we have 
    \[
        f_t^n(c(t)) = f_t^n \circ X_t \circ U_2(u(t)) = X_t \circ U_2 \circ g_t^n(u(t)).
    \]
    Since $(f_t^n)'|_{c(t)} \equiv 0$, by the chain rule $u(t)$ is a marked critical point of $X_t \circ U_2 \circ g_t^n$. 
    If $u(t)$ were preperiodic under $g_t$, then $c(t)$ would be preperiodic under $f_t$, contradicting the assumption that $c(t)$ is active. 
    Hence $u(t) \notin \Prep(g_t)$. 
    Because the set of critical points of $X_t \circ U_2$ is finite, there exists $s \in \N$ such that $u(t)$ is a marked critical point of $g_t^s$. 
    That is, choosing $s$ so that $g_t^s(u(t))$ avoids $\operatorname{Crit}(X_t \circ U_2)$ and $(X_t \circ U_2 \circ g_t^s)'|_{u(t)} = 0$, we obtain
    \[
        f_t^s(a(t)) = X_t \circ U_1 \circ g_t^s(u(t)),
    \]
    with $u(t)$ a marked critical point of $g^s_t$.

    Finally, by \cite{Ri23}, unless $g_t$ is conjugate to a power map, a Chebyshev polynomial, or a Latt\`es map (in particular unless $g_t$ is generalized Latt\`es), every rational function commuting with $g_t$ shares a common iterate with $g_t$. 
    Thus, $U_1$ shares a common iterate with $g_t$. 
    Enlarging $s$ if necessary so that $U_1^{m-1} = g_t^s$ for some $m > 0$, we conclude that $u(t)$ is a marked critical point of $U_1^m$ and
    \[
        f_t^s(a(t)) = X_t \circ U_1^m(u(t)).
    \]
    This completes the proof.
\end{proof}


 \begin{rmk}
     There is in fact additional information about $X_t$ in Theorem~\ref{thm: line-at-infinity-algebraic}. 
After introducing an appropriate orbifold structure on $\mathbb{P}^1$, 
the map $X_t$ becomes a Galois covering of $\mathbb{P}^1$. 
We omit a detailed discussion of this aspect here and refer the reader to \cite{Pa23} 
for a comprehensive treatment of orbifold structures and the associated covering map notation.
 \end{rmk}
 Notice that if we restrict to a constant family of regular polynomial endomorphisms, then Theorem \ref{thm: main-1} proved this degree stabilization statement under the assumption that $\mathcal{X}_s$ is a periodic curve under $F$ for a dense subset of $s \in S$.
 The following example shows that, in general, a family of curves can easily have degrees blow up under iteration by a family of regular polynomial endomorphisms. 
 \begin{example}
     Let $L_s = V(y - sx)$ be a family of lines parametrized by $s \in \C$ and let $\Phi : \A^1 \times \P^2 \to \A^1 \times \P^2$ be an endomorphism given by 
     $$ \Phi(s,p) =(s, F_s(p)) ,$$
     for $s \in \A^1$ and $p \in \P^2$, where for a fixed $\A^2$ chart we have $$F_s(x,y) = (x^2, y^2 + sx).$$

     Then we obviously have that 
     $$ F^n_s(L_s) \cap \A^2  = \{(u^{2^n} , s^{2^n}u^{2^n} + 2^{n-1}s^{2^{n}-1}u^{2^n -1} + o(u^{2^n -1})) : u \in \C\}.$$

     For any $s \in \C^*$, suppose there exists a polynomial $P_n(x,y) \in \C[x,y]$ so that $$V(P_n(x,y)) = \overline{F^n_s(L_s) \cap \A^2},$$ for a $n \in \N$. Then, after changing the coordinate by letting $z = x$ and $w = y - s^{2^n }x$, we have 
     $$\deg(P) \geq 2^n,$$
     since $\deg_u(z) = 2^n$, $\deg_u(w) = 2^n -1$ and 
     $$ P_n(z(u), w(u)) = 0$$
     for all $u \in \C$.
     Therefore, for any $s \in \C^*$, we have 
     $$ \deg(F^n_s(L_s)) \geq 2^n > 1 = \deg(L_s)$$
     for any positive integer $n$.
 \end{example}

\section{Endomorphisms with Infinitely Many Periodic Curves of Bounded Degree}\label{sect: classification}

In this section, let $F(x,y)$ denote a regular polynomial endomorphism of degree greater than $1$. 
We apply Theorem~\ref{thm: main-1} to classify regular polynomial endomorphisms that admit infinitely many periodic curves of bounded degree.

We adopt the notion of \emph{$k$-webs} from \cite{FP15}, and for the reader’s convenience we recall their definition following the same reference.

A \emph{$k$-web} $\mathcal{W}$ on a smooth complex surface $X$ is determined by a section 
\[
    \omega \in H^0\bigl(X, \Sym^k \Omega_X^1 \otimes \mathcal{N}_{\mathcal{W}}\bigr)
\]
for some line bundle $\mathcal{N}_{\mathcal{W}}$ on $X$, satisfying the following conditions:
\begin{enumerate}
    \item the zero set of $\omega$ has codimension~$2$ or is empty;
    \item for any point $p$ outside a proper hypersurface of $X$, after trivializing the line bundle $\mathcal{N}_{\mathcal{W}}$, the $k$-symmetric $1$-form $\omega(p) \in \Sym^k \Omega^1_{X,p}$ can be written as the product of $k$ pairwise distinct linear forms.
\end{enumerate}
Two sections $\omega$ and $\omega'$ determine the same $k$-web if and only if they differ by multiplication by a global nowhere-vanishing holomorphic function.

Intuitively, a $k$-web $\mathcal{W}$ on $X$ can be viewed as a multi-foliation consisting of $k$ distinct local leaves through a generic point of $X$. 
At a generic point, $\mathcal{W}$ is the superposition of $k$ pairwise transverse foliations $\mathcal{F}_1, \dots, \mathcal{F}_k$, and we may write
\[
    \mathcal{W} = \mathcal{F}_1 \boxtimes \cdots \boxtimes \mathcal{F}_k.
\]
An endomorphism $F$ of $X$ is said to \emph{preserve} a $k$-web $\mathcal{W}$ if it maps each local leaf of $\mathcal{W}$ to a leaf of $\mathcal{W}$.

We now show that if $F$ admits infinitely many periodic curves, then either there exists a preserved $k$-web with $k > 1$, or $F$ preserves a pencil of curves.

\begin{prop}\label{prop: infinite-per-curve-imply-k-web}
    Suppose that \( F \) admits an infinite set of irreducible periodic curves of degree bounded by some \( D \in \Z^+ \). Then either \( F \) preserves a \( k \)-web with \( k > 1 \), or \( F \) preserves a pencil of curves.
\end{prop}

\begin{proof}
    Without loss of generality, we may assume that \( F \) admits an infinite set of irreducible periodic curves all of degree exactly \( D \). 
    By Proposition~\ref{prop: 1-D-cases-summary}, there exists an irreducible subvariety 
    \( Z \subseteq \mathcal{M}_D \) parameterized by a morphism \( \tau : \P^1 \to Z \), 
    where \( Z \) is the closure in \( \mathcal{M}_D \) of an infinite subset of irreducible periodic curves of degree \( D \).
    Moreover, for any curve \( C \) with \( [C] \in Z \), we have \( [F(C)] \in Z \).
    For convenience, we also view \( \mathcal{M}_D \) as the space of (possibly reducible) curves of degree at most \( D \). 

    Our goal is to construct either a \( k \)-web or a pencil preserved by \( F \).

    We first show that the set
    \[
        \mathcal{S} \coloneqq 
        \{ p \in \P^2 : p \in C_1 \cap C_2,\; C_1 \neq C_2,\; C_1, C_2 \in Z \}
    \]
    is either Zariski dense in \( \P^2 \) or finite.

    Consider the space
    \[
        W = \overline{\{ (C_1, C_2, x) : C_1, C_2 \in Z,\; C_1 \neq C_2,\; x \in C_1 \cap C_2 \}} 
        \subseteq Z \times Z \times \P^2.
    \]
    For any point $x \in \P^2$, let $Z_x \coloneqq \{C \in Z : x \in C\}$. Let $\pi_3 : W \to \P^2$ onto the third factor. We have $\pi^{-1}_{3}(x)$ contains $(Z_x \times Z_x) \setminus \Delta$ which is of dimension $2\dim(Z_x)$. Hence $\dim (\pi_3^{-1}(x))$ is either $0$ or $2$.
    As any two distinct curves in \( Z \) intersect in finitely many points bounded in terms of \( D \),
    we have \( \dim(W) = 2 \dim(Z) = 2 \), since \( \dim(Z) = 1 \).
    Hence,
    \[
        \dim(\overline{\mathcal{S}}) = \dim(\pi_3(W)) \in \{0,2\},
    \]
    and thus \( \mathcal{S} \) is either Zariski dense in \( \P^2 \) or finite.

    Let us first suppose that $\mathcal{S}$ is finite.  
Define 
\[
\mathcal{O} \coloneqq \bigcap_{C \in Z} C.
\]
Since $\mathcal{O} \subseteq \mathcal{S}$, it is finite as well.  
Choose two sufficiently generic curves $C_1, C_2 \in Z$ such that 
\[
C_1 \cap C_2 = \mathcal{O}.
\]
Then, for any sufficiently generic $C_3 \in Z$ (avoiding the points in $\mathcal{S} \setminus \mathcal{O}$), Bézout’s theorem gives
\[
\sum_{p \in \mathcal{O}} (C_3, C_i)_p = D^2,
\]
where $i \in \{1,2\}$ and $D$ is the common degree of the generic curves in $Z$.  
Moreover, for every $p \in \mathcal{O}$ we have
\[
(C_1, C_3)_p = (C_2, C_3)_p,
\]
assuming that $C_1$, $C_2$, and $C_3$ are chosen generically in $Z$.

We now show that each such $C_3$ can be written in the form
\[
\lambda_0 C_1 + \mu_0 C_2
\]
for some $[\lambda_0 : \mu_0] \in \P^1_\C$, so that $C_3$ lies in the pencil of curves generated by $C_1$ and $C_2$.  
Suppose, on the contrary, that $C_3$ is not of this form. Then, for a generic choice of $[\lambda_0 : \mu_0] \in \P^1_\C$, there exists a point $q \notin \mathcal{O}$ such that 
\[
q \in C_3 \cap (\lambda_0 C_1 + \mu_0 C_2).
\]
In that case,
\[
(C_3, \lambda_0 C_1 + \mu_0 C_2)
 = \sum_{p \in \mathcal{O}} (C_3, \lambda_0 C_1 + \mu_0 C_2)_p
   + (C_3, \lambda_0 C_1 + \mu_0 C_2)_q
   > D^2,
\]
since for each $p \in \mathcal{O}$,
\[
(C_3, \lambda_0 C_1 + \mu_0 C_2)_p
   \ge \min\{ (C_3, C_1)_p, (C_3, C_2)_p \}.
\]
This contradicts Bézout’s theorem, which forces
\[
C_3 = \lambda_0 C_1 + \mu_0 C_2.
\]
Hence, $Z$ is a pencil of curves in $\P^2$ preserved by $F$.

    We now  assume $\mathcal{S}$ is Zariski dense and construct a \( k \)-web, with \( k \geq 2 \), using the family of curves in \( Z \).
    For each \( s \in \P^1 \), let \( C_s = V(P_s(x,y,z)) \) denote the corresponding curve,
    where \( P_s(x,y,z) \) is a homogeneous polynomial varying algebraically with \( s \).
    Define a differential form
    \[
        \omega([x:y:z]) \coloneqq 
        \left( \prod_{s : P_s(x,y,z) = 0} dP_s(x,y,z) \right)^{1/e},
    \]
    where \( e = \deg(\tau) \).
    Then \( \omega \) is a global section of 
    \( H^0(\P^2, \Sym^k \Omega_{\P^2}) \),
    where \( k \) is the number of curves in \( Z \) passing through a general point of \( \P^2 \)
    (so that \( ke = |\{ s \in \P^1 : P_s(x,y,z) = 0 \}| \) for general \([x:y:z]\)).

    We will construct from \( \omega \) a section in
    \[
        H^0(\P^2, \Sym^k \Omega_{\P^2} \otimes \mathcal{N}\mathcal{W}),
    \]
    for some line bundle \( \mathcal{N}\mathcal{W} \),
    defining a \( k \)-web preserved by \( F \).

    We first claim that the zero locus of \( \omega \) has codimension at least \( 2 \), unless empty.
    Suppose instead that \( \omega(p) = 0 \) for infinitely many points \( p \in \P^2 \).
    If this set is Zariski dense, then \( \omega \equiv 0 \), which implies that
    \( dP_s(x,y,z) = 0 \) for all \( [x:y:z] \in V(P_s) \) for infinitely many \( s \).
    Hence, each such \( C_s \) would consist entirely of singular points, forcing it to be reducible.
    This contradicts the fact that \( Z \) is generically irreducible,
    being the closure of infinitely many irreducible periodic curves.

    Next, suppose that the zero locus of \( \omega \) contains finitely many irreducible curves
    \( C_i = V(Q_i(x,y,z)) \) for \( 1 \leq i \leq \ell \).
    Then for each \( i \), there exists \( m_i \ge 1 \) such that \( Q_i^{-m_i}\omega \)
    vanishes only at finitely many points of \( C_i \).
    Define
    \[
        \omega' = \omega \prod_{i=1}^{\ell} Q_i^{-m_i},
    \]
    which can be viewed as a section of
    \[
        H^0\!\left(\P^2,\, \Sym^k \Omega_{\P^2} \otimes \mathcal{O}_{\P^2}\!\left(\sum_{i=1}^{\ell} m_i C_i\right)\!\right).
    \]
    Since \( \omega \) encodes the tangent directions of the family of curves \( Z \) preserved by \( F \),
    both \( \omega \) and \( \omega' \) are preserved by \( F \).

    It remains to show that \( \omega' \) can be locally written as a product of \( k \)
    distinct linear factors outside a proper Zariski closed subset of \( \P^2 \).
    Suppose the contrary: there exists a Zariski dense set of points \( p \)
    such that at \( p \) there are distinct curves \( C_s, C_t \in Z \)
    tangent to each other at \( p \).

    Let \( a_i(s) \in \C[s] \) denote the coefficients of \( P_s(x,y,z) \) for \( i = 1, \dots, h \), for some $h \in \N^+$,
    and define
    \[
        H(s,t) = \gcd\bigl\{ a_i(s)a_j(t) - a_j(s)a_i(t) : 1 \le i \ne j \le h \bigr\}.
    \]
    Write
    \[
        \nabla P_s(x,y,z) = (\partial_x P_s, \partial_y P_s, \partial_z P_s), \quad
        (A_{s,t}, B_{s,t}, C_{s,t}) = \nabla P_s \times \nabla P_t,
    \]
    so \( A_{s,t}, B_{s,t}, C_{s,t} \in \C[s,t][x,y,z] \).
    Consider the generalized resultant
    \[
        R([x:y:z]) = \operatorname{Res}_{s,t}\bigl((P_s - P_t)/H(s,t), P_s, A_{s,t}, B_{s,t}, C_{s,t}\bigr),
    \]
    whose zero locus consists of points where two distinct curves in \( Z \)
    are tangent to each other.
    Our assumption implies that \( R(p) = 0 \) for a Zariski dense set of \( p \),
    hence \( R \equiv 0 \).
    Therefore, at every \( p \in \P^2 \), there exist distinct \( C_s, C_t \in Z \)
    tangent to each other at \( p \).

    In particular, this holds for points \( p \in H_\infty \).
    If the intersections of the curves in \( Z \setminus \{H_\infty\} \) with \( H_\infty \)
    form a finite set, we reach a contradiction, since periodic curves are dense in \( Z \)
    but would then meet \( H_\infty \) only finitely often,
    whereas \( R \equiv 0 \) forces every point of \( H_\infty \) to lie on at least two such curves.
    On the other hand, if these intersections are infinite,
    then they occur at infinitely many periodic points of \( F|_{H_\infty} \).
    By Lemma~\ref{lem: one-point-or-web}, \( Z \) can be parameterized by a finite morphism
    \( \tau' : H_\infty \to Z \), implying that there is a non-superattracting periodic point
    of \( F|_{H_\infty} \) such that there is a unique curve in \( Z \setminus \{H_\infty\} \)
    passing through it which is a periodic curve, and the intersection is transverse (Lemma~\ref{lem: transvers-non-super-E}).
    This again contradicts \( R \equiv 0 \), which would require tangency at this point.
\end{proof}

Now we have established that any polynomial endomorphism admitting infinitely many periodic curves of bounded degree must preserve either a \( k \)-web with \( k > 1 \), or a pencil of curves. Consequently, the classification of endomorphisms of \( \P^2 \) that preserve \( k \)-webs, as given in \cite{FP15}, yields a corresponding classification of polynomial endomorphisms that possess infinitely many periodic curves of bounded degree.

We begin by recalling the classification results from \cite{FP15}.

\begin{thm}\label{thm: preserve-3-web}\cite[Theorem C (2)]{FP15}
    Suppose $\phi : \P^2 \to \P^2 $ is a regular polynomial endomorphism extended from $\A^2$ that preserves a $k$-web $\mathcal{W}$. Then there exists a projective toric surface $X$, and a finite group $G$ of automorphisms of $X$, such that the quotient space $X/G$ is isomorphic to $\P^2$. The image of $(\C^*)^2$ in $\P^2$ is a Zariski open subset $U$ which is totally invariant by $\phi$. Moreover, one can find an integer $d \geq 2$, and a collection of complex numbers $\lambda_1, \dots, \lambda_l$ such that the triple $(U, \phi^N, \mathcal{W})$ is the push forward under the natural quotient map of 
    $$ \left( (\C^*)^2, (x^d,y^d), \left( \lambda_1 \frac{dx}{x} + \frac{dy}{y}\right) \boxtimes \cdots \boxtimes \left(\lambda_l \frac{dx}{x} + \frac{dy}{y}\right)\right),$$
    with $N \in \{1,2,3\}$.
\end{thm}

\begin{thm}\label{thm: preserve-2-web}\cite[Theorem E]{FP15}
    Let $\phi : \P^2 \to \P^2$ be a holomorphic map of degree at least $2$ preserving a $2$-web $\mathcal{W}$. If $\phi$ does not preserve any other web then the pair $(\mathcal{W}, \phi)$ is one of the following: 
    \begin{enumerate}
        \item $\mathcal{W}$ is the union of two pencil of lines, and in suitable homogeneous coordinates $\phi(x,y) = (P(x), Q(y))$ or $(Q(y), P(x))$ where $P$ and $Q$ are polynomials of the same degree which are not conjugated to monomial maps nor to Chebyshev maps;

        \item $\mathcal{W}$ is the algebraic web dual to a smooth conic and $\phi$ is an Ueda map whose associated map on $\P^1$ is not a finite quotient of an affine map;

        \item $\mathcal{W}$ is the union of the pencil of lines $\{x/y = cst\}$ and the pencil of curves $\{x^py^q = cst\}$, with $p,q \in \N^*$, $\gcd\{p,q\} = 1$ and $\phi(x,y) = (x^d/R(x,y), y^d/R(x,y))$ where $R(x,y) = \prod^l_{i = 1}(1 + c_i x^py^q)$, $l(p+q) \leq d$, $c_i \in \C^*$ and the rational map $\theta \to \theta^d/(\prod_{i = 1}^l (1 + c_i \theta))^{p+q}$ is conjugated neither to a monomial nor to a Chebyshev map;
        \item $\mathcal{W}$ is the union of the pencil of lines $x/y = cst$ and the pencil of curves $xy = cst$, and $\phi(x,y) = (y^d/R(x,y), x^d/ R(x,y))$ where $R(x,y) = \prod^l_{i = 1}(1 + c_ixy)$, $2l \leq d$, $c_i \in \C^*$ and the rational map $\theta \to \theta^d/(\prod^l_{i = 1}(1 + c_i \theta))^2$ is conjugated neither to a monomial nor to a Chebyshev map;
        \item $\mathcal{W}$ is the quotient of the foliation $[dx]$ on $\P^1 \times \P^1$ by the action of $(\Z/2)^2$ generated by $(x,y) \to (-x,-y)$ and $(x,y) \to (y,x)$; and $\phi$ is the quotient of a polynomial endomorphism of $\P^1 \times \P^1$ of the form $(f(x), f(y))$ where $f$ is an odd or even polynomial which is conjugated neither to a monomial nor to a Chebyshev map;
        \item $\mathcal{W}$ is the quotient of the foliation $[dx]$ on $\P^1 \times \P^1$ by the action of $(\Z/2)^3$ generated by $(x,y) \to (y,x)$, $(x,y) \to (x^{-1}, y^{-1})$, and $(x,y) \to (-x,-y)$; and $\phi$ is the quotient of an endomorphism of $\P^1 \times \P^1$ of the form $(f(x), f(y))$ where $f$ is a rational map that commutes with the group generated by $x \to -x$ and $x \to x^{-1}$ which is not a finite quotient of an affine map.
    \end{enumerate}
    
\end{thm}


Together, these two theorems provide a complete classification of regular polynomial endomorphisms that preserve a \( k \)-web with \( k \ge 2 \). The following theorem from \cite{DJ07} gives the classification of endomorphisms of \( \P^2 \) that preserve a pencil of curves.

\begin{thm}\label{thm: preserving-pencil-curves}\cite{DJ07}
    There are two types of irreducible pencils $\mathcal{P}$ invariant under an endomorphism $F$ of $\P^2$.
    \begin{enumerate}
        \item $\mathcal{P}$ is an elementary pencil, that is, the pencil of lines through a point. Under a suitable homogeneous coordinates on $\P^2$, $F$ takes the form $F[x:y:z] = [P(x,y) : Q(x,y):R(x,y,z)]$.
        \item $\mathcal{P}$ is a binomial pencil. Under a suitable homogeneous coordinates on $\P^2$, $F$ takes the form $F[x:y:z] = [x^d:y^d: R]$ or $F[x:y:z] = [y^d:x^d:R]$ for $R(x,y,z) = z^{d - kl} \prod^l_{i = 1}(z^k + c_ix^hy^{k-h})$, where $k \geq 2$, $0 \leq l \leq d/k$, $c_i \in \C^*$, $0 < h< k$ and $\gcd(h,k) = 1$. 
    \end{enumerate}
\end{thm}

Since we have shown in Proposition~\ref{prop: infinite-per-curve-imply-k-web} that a regular polynomial endomorphism admitting infinitely many periodic curves of bounded degrees must preserve either a $k$-web with $k \ge 2$ or a pencil of curves, Theorems~\ref{thm: preserve-3-web}, \ref{thm: preserve-2-web}, and \ref{thm: preserving-pencil-curves} together provide a complete classification of the possible polynomial endomorphisms.

It then follows directly that Theorem~\ref{thm: classification-summary} gives a coarse summary of all possible situations.

\begin{proof}[Proof of Theorem~\ref{thm: classification-summary}]
By Proposition~\ref{prop: infinite-per-curve-imply-k-web}, the map $F$ must belong to one of the classes described in Theorems~\ref{thm: preserve-2-web}, \ref{thm: preserve-3-web}, or \ref{thm: preserving-pencil-curves}.  

Suppose first that $F$ preserves a pencil of curves, so that $F$ falls into one of the cases of Theorem~\ref{thm: preserving-pencil-curves}.  
If it is in the first case, then there exists a pencil of lines passing through a single common point $p \in \P^2$ that are preserved by $F$.  
If $p$ is not contained in the totally invariant line at infinity, then after a change of coordinate, we assume $p = (0,0)$ and $F$ preserves a pencil of lines through $(0,0)$, which implies that $F$ is homogeneous.  
If $p \in H_\infty$, after a suitable change of coordinate, we have the pencil of lines contains $\{V(x - k) : k \in \C\}$. This implies that it is a regular polynomial skew product as it sends vertical lines to vertical lines.  
If $F$ falls in the second case of Theorem~\ref{thm: preserving-pencil-curves}, it is again clear that $F$ becomes a regular polynomial skew product after taking an iterate.

Next, if $F$ belongs to the class of Theorem~\ref{thm: preserve-3-web}, then there exists a generically finite rational map 
\[
    \mu: \P^1 \times \P^1 \longrightarrow \P^2,
\]
which restricts to a morphism on $(\C^*)^2$ arising from a finite quotient, such that
\[
    \mu \circ (x^d, y^d) = F \circ \mu.
\]

Finally, if $F$ is in one of the classes of Theorem~\ref{thm: preserve-2-web}, then in subcases (3) and (4), $F$ becomes a regular polynomial skew product after replacing it with an iterate and choosing a suitable affine coordinate on some $\A^2 \subseteq \P^2$.  
In all other cases, after possibly taking a further iterate, there exists a finite morphism
\[
    \mu: \P^1 \times \P^1 \longrightarrow \P^2,
\]
induced by a finite quotient, satisfying
\[
    \mu \circ (f, g) = F \circ \mu,
\]
where $(f, g)$ is a split endomorphism of $\P^1 \times \P^1$.  
In cases (1) and (5), both $f$ and $g$ are polynomials.  
In cases (2) and (6), we have
\[
    \mu \circ (f, f) = F \circ \mu
\]
for some rational function $f$, and we now show that $f$ must in fact be a polynomial after being replaced with a proper iterate.  

Since $F$ is a regular polynomial endomorphism, the line at infinity $H_\infty$ is totally invariant under $F$.  
Let $$E \coloneqq \mu^{-1}(H_\infty) \subseteq \P^1 \times \P^1.$$
since $H_\infty$ is totally invariant under $F$. This gives that $E$ is totally invariant under $(f,f)$. After replacing $(f,f)$ and $F$ with some iterate, we may assume that there exists an irreducible component of $E$ that is totally invariant under $(f,f)$. Let's abuse notation to still call it $E$. 

If $E$ is a vertical or horizontal line, then we conclude directly that $f$ is a polynomial after being replaced with some iterate. Suppose now that $E$ projects dominantly to both factors. Let $p$ be a fixed point of $f$ and let $L_p \coloneqq V(x - p)$. Then 
\[(f,f)|_{L_p}^{-1}(L_p \cap E) = L_p \cap (f,f)^{-1}(E) = L_p \cap E.\]
Therefore $L_p \cap E$ is a finite exceptional set of $(f,f)|_{L_p} = f$. Thus, $f$ is a polynomial after an iteration.
\end{proof}

Here we provide an example of a regular polynomial endomorphism of $\P^2$ admitting infinitely many periodic curves, corresponding to the third case of Theorem~\ref{thm: classification-summary}.

\begin{example}
Consider the regular polynomial endomorphism on $\A^2$
\[
    A_2(x, y) = (x^2 - y^2 - 2x,\, 2xy + 2y),
\]
which arises as the \emph{folding map} associated to the Lie algebra $\mathcal{A}_2$ (see~\cite{Sil24} for details).  
It admits non-trivial commuting pairs; indeed, it commutes with
\[
    A_3(x, y) = (x^3 - 3xy^2 - 3x^2 - 3y^2 + 3,\, 3x^2y - y^3),
\]
the folding map constructed from the Lie algebra $\mathcal{A}_3$.

Under the change of coordinates
\[
    z = x + i y, \qquad w = x - i y,
\]
the map becomes conjugate to
\[
    A_2(z, w) = (z^2 - 2w,\, w^2 - 2z),
\]
and we will work with this form for convenience.

One can verify that for any $k \in \C^*$ that is periodic under $x \mapsto x^2$, the curve
\[
    V(y - kx - k^{-1} + k^{-2})
\]
is periodic under $A_2$.  

Moreover, define the generically finite rational map
\[
    \Phi(x, y) = (x + y + x^{-1}y^{-1},\, x^{-1} + y^{-1} + xy)
\]
from $\P^1 \times \P^1$ to $\P^2$.  
Then the following commutative diagram holds:
\[
\begin{tikzcd}
\P^1 \times \P^1 \arrow{r}{(x^2,\, y^2)} \arrow[swap]{d}{\Phi} & \P^1 \times \P^1 \arrow{d}{\Phi} \\
\P^2 \arrow{r}{A_2} & \P^2.
\end{tikzcd}
\]

For a generic point in $\A^2$, there are three curves in the family
\[
    \mathcal{C} = \{\, V(y - kx - k^{-1} + k^2) \mid k \in \C^* \,\}
\]
passing through it.  Notice that $\mathcal{C}$ is preserved by $A_2$, since
\[
    A_2\big(V(y - kx - k^{-1} + k^2)\big) = V(y - k^2x - k^{-2} + k^4).
\]
Thus, $\mathcal{C}$ forms a $3$-web preserved by $A_2$, represented by the global differential
\[
    \omega = \prod_{k:\, y - kx - k^{-1} + k^2 = 0} (-k\,\mathrm{d}x + \mathrm{d}y + (k^2 - k^{-1})\,\mathrm{d}z).
\]
Substituting the symmetric functions of the roots in terms of $x$ and $y$, this can be expressed as
\begin{align*}
 \omega &= (\mathrm{d}y - y\,\mathrm{d}z)^3 
 + x\,(\mathrm{d}y - y\,\mathrm{d}z)^2(-\mathrm{d}x + x\,\mathrm{d}z) \\
 &\quad + y\,(\mathrm{d}y - y\,\mathrm{d}z)(-\mathrm{d}x + x\,\mathrm{d}z)^2 
 + (-\mathrm{d}x + x\,\mathrm{d}z)^3.
\end{align*}

Now, viewing $\Phi$ as a map from $(\C^*)^2$ to $\P^2$, we have
\[
    \Phi^{-1}\big(V(y - kx - k^{-1} + k^2)\big)
    = V(y - k^{-1}x) \cup V(x - k) \cup V(y - k),
\]
and therefore
\[
    \Phi^*\omega = \mathcal{F}_1 \boxtimes \mathcal{F}_2 \boxtimes \mathcal{F}_3,
\]
where $\mathcal{F}_1$, $\mathcal{F}_2$, and $\mathcal{F}_3$ are the foliations on $(\C^*)^2$ corresponding respectively to the families 
$V(y - k^{-1}x)$, $V(x - k)$, and $V(y - k)$ as $k$ ranges over $\C^*$.  
This is precisely the phenomenon described in Theorem~\ref{thm: preserve-3-web}.
\end{example}

\section*{ Acknowledgements}
The author gratefully acknowledges Prof. Junyi Xie for the insightful discussions and generous hospitality during a short visit to the Beijing International Center for Mathematical Research. Special thanks are also due to Prof. Niki Myrto Mavraki for the valuable conversations regarding this project and her work with Prof. DeMarco on "Geometry of Preperiodic Points", particularly during the "Workshop on Model Theory, Algebraic Dynamics, and Differential-Algebraic Geometry." The author thanks the Fields Institute and the workshop organizers for hosting this enriching event. Many thanks to Prof. Charles Favre for answering questions on k-webs, and to Profs. Jason Bell, Junyi Xie, and Niki Myrto Mavraki for their comments on an earlier draft of this paper. The author was supported in part by NSERC grant RGPIN-2022-02951.


\end{document}